\documentclass[a4paper,12pt]{article} %{article}
\title{Extending the double ramification cycle by resolving the Abel-Jacobi map}

\usepackage{amsmath, amssymb, mathrsfs, amsthm, shorttoc, geometry}% amsfonts, url,, tikz, wrapfig, newclude, }
\usepackage{mathtools} % for \coloneqq (and presumably other stuff!)
\usepackage{hyperref}
\usepackage[all]{xy}
\usepackage{enumitem}

\usepackage{cleveref}
\let\oref\ref
\AtBeginDocument{\renewcommand{\ref}[1]{\cref{#1}}}

\usepackage{pgf,tikz}
\usetikzlibrary{arrows}
\usepackage{tikz-cd} %this is the old version, but it also the one the arXiv supports at the moment (2016)

\newcommand{\on}[1]{\operatorname{#1}}
\newcommand{\bb}[1]{{\mathbb{#1}}}
\newcommand{\cl}[1]{{\mathscr{#1}}}
\newcommand{\ca}[1]{{\mathcal{#1}}}

\newcommand{\ul}[1]{{\underline{#1}}}

\newcommand{\abs}[1]{\lvert#1\rvert}

\newcommand{\lra}{\longrightarrow}
\newcommand{\hra}{\hookrightarrow}
\newcommand{\sub}{\subseteq}
\newcommand{\tra}{\rightarrowtail}
\newcommand{\iso}{\stackrel{\sim}{\lra}}

\theoremstyle{definition}
\newtheorem{definition}{Definition}[section]
\newtheorem{conjecture}[definition]{Conjecture}

\theoremstyle{plain}% default
\newtheorem{proposition}[definition]{Proposition}
\newtheorem{lemma}[definition]{Lemma}
\newtheorem{theorem}[definition]{Theorem}
\newtheorem{corollary}[definition]{Corollary}

\theoremstyle{remark}

\newtheorem{remark}[definition]{Remark}

\renewcommand{\phi}{\varphi}

\newcommand{\loz}{\lozenge}
\newcommand{\bloz}{\blacklozenge}
\author{David Holmes}
\date{\today}

\newcounter{nootje}
\newcommand\todo[1]{[*\thenootje]\marginpar{\tiny\begin{minipage}
{20mm}\begin{flushleft}\thenootje : 
#1\end{flushleft}\end{minipage}}\addtocounter{nootje}{1}}
\newcommand{\expanded}[1]{#1}
\renewcommand{\expanded}[1]{}

\newcommand{\beq}{\begin{equation}}%depreciated
\newcommand{\eeq}{\end{equation}}%depreciated
\newcommand{\beqs}{\begin{equation*}}%depreciated
\newcommand{\eeqs}{\end{equation*}}%depreciated

\begin{document}
\maketitle
\begin{abstract} 
Over the moduli space of smooth curves, the double ramification cycle can be defined by pulling back the unit section of the universal jacobian along the Abel-Jacobi map. This breaks down over the boundary since the Abel-Jacobi map fails to extend. We construct a `universal' resolution of the Abel-Jacobi map, and thereby extend the double ramification cycle to the whole of the moduli of stable curves. In the non-twisted case, we show that our extension coincides with the cycle constructed by Li, Graber, Vakil via a virtual fundamental class on a space of rubber maps. 

\emph{Keywords: curves, jacobians, cycles}
%We provide a `Riemann-Roch' formula for our class, which we plan to make more explicit in future work. 
%
%Only changes: `Stack of weighted stable curves', `Weightings on a graph', added a little more decoration to the graph (vertices now need genera), and made corresponding changes to the definitions. 
%
%Section 4, proof of lemma 4.3: changed first instance of $\ca O$ to an $\omega^{\otimes k}$. No other changes to proof. 
%
%Section 5, `the universal line bundle', formulae \ref{eq:simple_f_formula} and \ref{eq:triv_hom} slightly adjusted. Proof of \ref{thm:properness_lozenge}, changed formula \ref{iso_3} and \ref{eq:formula_bar_f}. 
%
%Section 7, added remark that it only applies if $k=0$. 
\end{abstract}

%\shorttableofcontents{Contents}{1}

\tableofcontents

\newcommand{\Mtildes}{ \widetilde{\ca M}^\Sigma}
\newcommand{\Mhat}{\hat{\cl M}}
\newcommand{\sch}[1]{\textcolor{blue}{#1}}

%new:
\newcommand{\Mbar}{\overline{\ca M}}
\newcommand{\MD}{\ca M^\blacklozenge}
\newcommand{\Md}{\ca M^\lozenge}

%\newgeometry{top=2cm}

\section{Introduction}

%{Bas on the intro: You are writing to teh managing editor, your aim is to get them to send it to a referee. For a 29 paper paper looking at about 2.5 pages or so of intro (so maybe 10\%?). At least in that space the important stuff should appear. Editors look out for phrases like `In this paper, we...'. Make clear the relation/distinction to previous work. 
%
%In this case should mention the new paper of MW. }

Fix integers $g$, $n\ge 0$ satisfying $2g - 2 + n >0$, and integers $a_1 \dots a_n$ and $k$ such that $\sum_i a_i = k(2g-2)$. Over the moduli stack $\ca M_{g,n}$ we have the universal curve $C_{g,n}$ with tautological sections $x_1, \dots, x_n$. Write $J$ for the universal jacobian (an abelian scheme over $\ca M_{g,n}$) and $\sigma$ for the section of $J$ given by the line bundle $\omega^{\otimes k}(-\sum_i a_i x_i)$ on $C_{g,n}$. The pullback of the unit section along $\sigma$ defines a codimension-$g$ cycle class on $\ca M_{g,n}$, the \emph{double ramification cycle} (DRC). The problem of producing a `reasonable' extension of the DRC to the Deligne-Mumford-Knudsen compactification $\Mbar_{g,n}$, and computing the class of the resulting cycle in the tautological ring, was proposed by Eliashberg. 

In the case $k=0$, an extension of the class to the whole of $\Mbar_{g,n}$ was constructed by Li \cite{Li2001Stable-morphism}, \cite{Li2002A-degeneration-}, Graber, and Vakil \cite{Graber2005Relative-virtua}. This class was computed in the compact-type case by Hain \cite{Hain2013Normal-function}, and this was extended to treelike curves with one loop by Grushevsky and Zakharov \cite{Grushevsky2012The-zero-sectio}. More recently, Janda, Pandharipande, Pixton and Zvonkine \cite{Janda2016Double-ramifica} computed this class on the whole of $\Mbar_{g,n}$, proving a conjecture of Pixton. 

A construction for arbitrary $k$ was proposed by Gu\'er\'e \cite{Guere2016A-generalizatio}, but the situation here is more complicated. Pixton's conjecture makes sense for all $k$, but is purely combinatorial in origin. A more geometric conjecture is given by Janda, Pandharipande, Pixton, and Zvonkine in the appendix of \cite{Faber2005Relative-maps-a} for $k=1$, generalised by Schmitt \cite{Schmitt2016Dimension-theor} to all $k \ge 1$; their formulae are moreover conjectured to coincide with Pixton's. However, it is at present unclear whether Gu\'er\'e's construction is compatible with these conjectures. 

%However, the author has learnt from Schmitt that for $k \ge 2$ the conjectured formula seems incompatible with the construction of Gu\'er\'e, so more work needs to be done to understand this. 

In this paper we give a simple construction of an extension of the DRC for arbitrary $k$, and for $k=0$ we verify that it agrees with the construction of Li, Graber and Vakil. Given the situation described above, it seems very interesting to know whether it coincides with the construction of Gu\'er\'e when $k$ is arbitrary; this is the subject of current work in progress. Our construction is valid in arbitrary characteristic, and produces a class in $\on{CH}(\Mbar_{g,n})$; we do not need to work with rational coefficients. 

Kass and Pagani \cite{Kass2017The-stability-s} have recently constructed large numbers of extensions of the DRC (for every $k \in \bb Z$). In \cite{Holmes2017Jacobian-extens} we show that certain of their classes coincide with the one constructed in this article.

%The problems of constructing and computing extensions of the DRC for all $k$ which agree with the standard constructions for $k=0$ seem still to be open. In this paper we solve the first of these problems; more precisely, we give a construction valid for arbitrary $k$, and verify when $k=0$ that it coincides with the cycle of Li, Graber, and Vakil. 

\subsection*{Statement of main results}

We write $J/ \Mbar_{g,n}$ for the unique semiabelian extension of the universal jacobian (also denoted above by $J$), and $\omega$ for the relative dualising sheaf of the universal stable curve over $\Mbar_{g,n}$. The section $\sigma = \omega^{\otimes k}\left(-\sum_i a_i x_i\right)$ does not in general extend over the whole of $\Mbar_{g,n}$. This can be partially resolved by blowing up $\Mbar_{g,n}$. Let $f\colon X \to \Mbar_{g,n}$ be a proper birational map from a normal stack (a `normal modification'). The section $\sigma$ is then defined on some dense open of $X$. We write $\mathring X$ for the largest open of $X$ on which this rational map can be extended to a morphism, and $\sigma_X\colon \mathring X \to J$ for the extension. 

We define the \emph{double ramification locus} $\on{DRL}_X \tra \mathring X$ to be the schematic pullback of the unit section of $J$ along $\sigma_X$, and the \emph{double ramification cycle} $\on{DRC}_X$ to be the cycle-theoretic pullback, \emph{as a cycle supported on }$\on{DRL}_X$ (actually we pull back the unit section in $J \times_{\Mbar_{g,n}} \mathring X$ along the induced map $\mathring  X \to J \times_{\Mbar_{g,n}} \mathring X$, since the latter is a regular closed immersion). Now the map $\mathring X \to \Mbar_{g,n}$ is rarely proper, but the map $\on{DRL}_X\to \bar {\ca M}_{g,n}$ \emph{is} quite often proper. More precisely, we have: 

\begin{theorem}[\ref{{main_theorem:proper_full}}]\label{main_theorem:proper}
In the directed system of all normal modifications of $\Mbar_{g,n}$, those $X$ such that $\on{DRL}_X \to \Mbar_{g,n}$ is proper form a cofinal system. 
\end{theorem}

Now $\on{DRC}_X$ is supported on $\on{DRL}_X$, so when the map $\on{DRL}_X \to \Mbar_{g,n}$ is proper, we can take the pushforward of $\on{DRC}_X$ to $\Mbar_{g,n}$. Writing ${\pi_X}_*\on{DRC}_X$ for the resulting cycle on $\Mbar_{g,n}$, we might hope that these cycles `converge' in some way as we move up in the tower of modifications. In fact, this is true in a very strong sense: 

\begin{theorem}[\ref{main_theorem:limit_full}]\label{main_theorem:limit}
The net ${\pi_X}_*\on{DRC}_X$ is eventually constant in the Chow ring  $\on{CH}(\Mbar_{g,n})$. We denote the limit by $\overline{\on{DRC}}$. 
\end{theorem}
In other words, there exists a normal modification $X \to \Mbar_{g,n}$ such that, for every normal modification $X'$ dominating $X$, we have ${\pi_{X}}_*\on{DRC}_{X}  ={\pi_{X'}}_*\on{DRC}_{X'} $. Since every two normal modifications can be dominated by a third, this implies the existence of a uniquely determined `limit' of this collection of ${\pi_{X'}}_*\on{DRC}_{X'} $, and we denote this limit by $\overline{\on{DRC}}$.

%cycle $\overline{\on{DRC}}$ in $\on{CH}(\Mbar_{g,n})$ such that, for every normal modification $X \to \Mbar_{g,n}$, there exists a normal modification $X'$ dominating $X$ such that ${\pi_{X'}}_*\on{DRC}_{X'} = \overline{\on{DRC}}$. 

\begin{theorem}[\ref{main_theorem:comparison_full}]\label{main_theorem:comparison}
Suppose that $k=0$. The class $\overline{\on{DRC}}$ coincides with the extension of the double ramification cycle defined by Li, Graber, and Vakil \cite{Graber2005Relative-virtua} (and computed by \cite{Janda2016Double-ramifica}). \end{theorem}

%We are currently engaged (jointly with Alessandro Chiodo) in computing the class of $\overline{\on{DRC}}$ in the tautological ring. When $g=1$ and $k=0$ this computation is very simple, and (reassuringly) some manipulation shows that our formula is equivalent to the one obtained by \cite{Janda2016Double-ramifica} (which was already known to Hain \cite{Hain2013Normal-function}). 

\subsection*{Strategy of proof}
The proofs of \ref{main_theorem:proper,main_theorem:limit} proceed by constructing a `universal' stack over $\Mbar_{g,n}$ making the map $\sigma$ extend. More formally, we call a morphism $t\colon T \to \Mbar_{g,n}$ \emph{$\sigma$-extending} if the section $\sigma$ can be extended to $J$ after pulling back to $T$. For this to make sense we need the pullback of $\ca M_{g,n}$ to be dense in $T$, and we also require $T$ to be normal and to admit a resolution of singularities (the latter being automatic in characteristic zero by \cite{Hironaka1964Resolution-of-s}). 
%A precise statement is given in \ref{}; these conditions are important for our application to the double ramification cycle, see \ref{}. . 

Our \emph{universal resolution of the Abel-Jacobi map} will then be a terminal object in the category of $\sigma$-extending morphisms. In \ref{cor:universal_property_md} we prove existence; we denote this object $\Md_{g,n}$. The section $\sigma$ extends over $\Md_{g,n}$ by definition, and moreover $\Md_{g,n}$ is the `universal' stack over $\Mbar_{g,n}$ over which $\sigma$ extends. Note that $\Md_{g,n}$ depends non-trivially on the ramification data $a_1, \dots, a_n$. 

The proof of existence of $\Md_{g,n}$ is constructive, and equips it with a logarithmic structure making it logarithmically \'etale over $\Mbar_{g,n}$. The construction is given locally using toric geometry; we write down explicit combinatorial recipes for the fans of toric varieties, and glue them to build $\Md_{g,n}$. The combinatorial recipe is unavoidably rather complicated, but is amenable to implementation on a computer; in this way we produced \ref{fig:fan_Md} showing one of these fans. 

In \ref{sec:universal_bundle} we also give an explicit description of the universal line bundle on the universal curve over $\Md_{g,n}$ in terms of this toric data, and in \ref{sec:proper_general_case} we use this to prove the key properness result needed for \ref{main_theorem:proper} (though in characteristic zero there is a simpler proof, see \ref{sec:char_0_properness}). 

The morphism $\Md_{g,n} \to \Mbar_{g,n}$ may be viewed as a universal resolution of the indeterminacies of the Abel-Jacobi map. This solves a problem proposed by Grushevsky and Zakharov \cite[Remark 6.3]{Grushevsky2012The-zero-sectio}. Together with our other results, it seems also to solve a problem of Cavalieri, Marcus and Wise proposed in \cite[Section 1.4]{Cavalieri2011Polynomial-fami}. 

To conclude the proofs of the main results, we will choose a suitable compactification $\MD_{g,n}$ of $\Md_{g,n}$. Now if $X \to \Mbar_{g,n}$ is any normal modification which factors via a map $f\colon X \to \MD_{g,n}$, we will verify that
\begin{equation*}
\mathring X = f^{-1} \Md_{g,n}, \,\,\,\;\; \on{DRL}_X = f^{*}\on{DRL}_{\MD_{g,n}} \text{, and } f_*\on{DRC}_X = \on{DRC}_{\MD_{g,n}}. 
\end{equation*}
(here $\mathring X$ still denotes the largest open of $X$ on which the rational map $\sigma$ can be extended to a morphism). 
We will establish that $\on{DRL}_\loz$ is proper over $\bar{\ca M}_{g,n}$, whereupon the analogous properness result will hold for any normal modification which factors via $\MD_{g,n}$, establishing \ref{main_theorem:proper}. \Cref{main_theorem:limit} will then follow fairly formally. Finally, when $k=0$ we use the deformation-theoretic tools of Marcus, Wise and Cavalieri to establish \ref{main_theorem:comparison}.

\subsubsection*{A simple formula for the double ramification cycle}

Suppose we work over a field of characteristic zero, or perhaps more generally over a ring $R$ such that resolution of singularities is known for all schemes of finite type over $R$. Then there is a simpler way to prove the \emph{existence} of the universal $\sigma$-extending morphism: the normalisation of the closure of the schematic image of $\sigma$ in $J$ satisfies the universal property (details are given in \ref{sec:char_0_properness}). This does not tell us a huge amount since we have no explicit description of this closure; for example, it is a-priori far from clear that it admits a log structure making it log etale over $\Mbar$. However, it does allow us to simplify the proof of \ref{main_theorem:proper}; details are given in \ref{sec:char_0_properness}. 

This allows us to give a very simple formula for the extension of the double ramification cycle. We define $\on{DR}_{naive}$ to be the cycle on $\Mbar_{g,n}$ obtained by pulling back the schematic image of $\sigma$ along the unit section. Still in characteristic zero, a small argument with the projection formula (of which details can be found in \cite{Holmes2017Jacobian-extens}) shows that the resulting cycle is equal to the one constructed in this paper, and hence also to that of Li, Graber and Vakil. 

Note that we do need resolution of singularities since, in order to verify that the normalisation of the closure of $\sigma$ is $\sigma$-extending, we need in particular that it admits a (local) resolution of singularities, which is in general far from clear. Of course, one could drop from the definition of $\sigma$-extending that $T$ admit a local resolution of singularities, but then our methods break down (in particular we are no longer able to prove the critical \ref{lem:w_aligned_implies_extending}), so we can no longer give an explicit description of the universal object, or equip it with a natural logarithmic structure etc.

\subsection*{Conjectural relationship to a cycle of Pixton}

Given the data $g$, $n$, $k$ and $\ul a$, Pixton introduced a cycle $P^{g,k}_g(A)$ in the tautological ring of $\Mbar_{g,n}$, given in terms of decorated graphs  --- here $A = (a_1 + k, \dots, a_n + k)$; the details of the construction can be found in \cite[section 1.1]{Janda2016Double-ramifica}. The main result of \cite{Janda2016Double-ramifica} shows that, when $k=0$, there is an equality of cycles 
\begin{equation*}
\overline{\on{DRC}}_{\ul a} = 2^{-g}P^{g,0}_g(A)
\end{equation*}
(here we write $\overline{\on{DRC}}_{\ul a}$ to make the dependence on $\ul a = (a_1, \dots, a_n)$ explicit). 
Now that we have a construction of $\overline{\on{DRC}}_{\ul a}$ valid for all $k$, it seems natural to propose
\begin{conjecture}\label{conjecture:pixton}
For all $k$, there is an equality of cycles 
\begin{equation*}
\overline{\on{DRC}}_{\ul a} = 2^{-g}P^{g,k}_g(A)
\end{equation*}
in $\on{CH}^g(\Mbar_{g,n})$. 
\end{conjecture}
Some evidence for this conjecture is given in the following section.

\subsection*{Multiplicativity of the double ramification cycle}
In \cite{Holmes2017Multiplicativit} we use the results of this paper to construct an extension of the double ramification cycle to the small $b$-Chow ring of $\Mbar_{g,n}$ --- this is the colimit of the Chow rings of the smooth blowups of $\Mbar_{g,n}$, with transition maps given by pulling back cycles. Note that the `asymptotic' approach we adopt here is essential for this construction. Given vectors $\ul a$, $\ul b$ of ramification data, we will show that the basic multiplicativity relation 
\begin{equation}\label{eq:mult}
\on{DRC}_{\ul a} \cdot \on{DRC}_{\ul b} = \on{DRC}_{\ul a} \cdot \on{DRC}_{\ul a + \ul b} 
\end{equation}
holds in the small $b$-Chow ring of $\Mbar_{g,n}$, but fails in the Chow ring of $\Mbar_{g,n}$. 

One consequence of this is that, if \ref{conjecture:pixton} is true, the relation (\oref{eq:mult})
should also hold for Pixton's cycles on the locus of compact type curves. This relation can be independently checked using known relations in the tautological ring (see \cite{Holmes2017Multiplicativit}), which may be seen as evidence for \ref{conjecture:pixton}. 

\subsection*{Conjectural relationship to a cycle of Janda, Pandharipande, Pixton, and Zvonkine}

In the appendix of \cite{Farkas2016The-moduli-spac}, Janda, Pandharipande, Pixton, and Zvonkine define a cycle $H_{g,a}$ tautological ring of $\Mbar_{g,n}$ in the case where $k=1$ and at least one $a_i<0$, and conjecture that $H_{g,a}$ coincides with Pixton's cycle $P_g^{g,k}(A)$. We are currently engaged (jointly with Johannes Schmitt) in verifying the equality $H_{g,a} = \overline{\on{DRC}}_{\ul a}$, which may be viewed as a step towards \ref{conjecture:pixton}, or towards the conjecture of Janda, Pandharipande, Pixton, and Zvonkine, or both.

\subsection*{Comparison to the approach of Li, Graber and Vakil}

%The essential difficulty in extending the above definition of the DRC over the boundary of $\Mbar_{g,n}$ is the lack of a suitable model of the jacobian $J$ over the boundary. The entire Picard functor is too large (it is not separated, so intersections behave badly). The identity component of the Picard functor is a-priori too small (the section $\sigma$ fails to extend). In this paper we resolve the indeterminacies of $\sigma$ by blowing up the boundary of $\Mbar_{g,n}$; we give a construction of a suitable blowup, and also show that every blowup dominating it works in the same way (giving an interpretation of the DRC in terms of b-divisors). These results are described in more detail in the next section. 

The approach of Li \cite{Li2001Stable-morphism}, \cite{Li2002A-degeneration-}, Graber, and Vakil \cite{Graber2005Relative-virtua} when $k=0$ is based on thinking of the DRC as the locus of curves admitting a map to $\bb P^1$ with specified ramification over $0$ and $\infty$. They define a stack of stable maps to `rubber $\bb P^1$', i.e. to $[\bb P^1 / \bb G_m]$. They then define the DRC as the pushforward of a virtual fundamental class from this stack of stable maps. This enables them to apply the well-developed machinery of virtual classes and spaces of stable maps. In contrast, our approach is in a sense very naive; using blowups to resolve the indeterminacies of rational maps goes back to classical algebraic geometry (and in the non-proper case to Raynaud and Gruson \cite{Raynaud1971Criteres-de-pla}). The more elementary nature of our approach makes it very easy to extend to the case $k \neq 0$, and we hope will allow further extensions; we are particularly interested in developing further the Gromov-Witten theory of $B \bb G_m$, extending the results of \cite{Frenkel2009Gromov-Witten-g} beyond the `admissible' case. %One by-product of our approach is a natural `double ramification locus' on which the double ramification cycle is supported. 

\subsection*{Comparison to the cycle of Kass and Pagani}

The preprint \cite{Kass2017The-stability-s} of Kass and Pagani (posted on the same day as the first version of this preprint) gives a 
very different approach to resolving the Abel-Jacobi map, by studying families of stability conditions on the space of rank 1 torsion-free sheaves. Moving through a suitable family produces a series of flips of a certain compactified jacobian, after which the Abel-Jacobi map extends over $\overline{\mathcal{M}}_{g,n}$ (this series of flips depends on the ramification data). In essence, we modify the source of the Abel-Jacobi map, whereas Kass and Pagani modify the target. In this way, they produce a number of different extensions of the DRC. In \cite{Holmes2017Jacobian-extens} we show that, for certain choices of stability conditions, the resulting double ramification cycle coincides with that given in this article. 

\subsection*{Comparison to some other recent results}

%; it remains to be understood how these extensions are related to one-another, and to the DRC of Li, Graber and Vakil. 

More recently, Marcus and Wise \cite{Marcus2017Logarithmic-com} have given another approach to resolving the Abel-Jacobi map when $k=0$, rather closer in spirit to the present preprint. They also use logarithmic geometry to modify $\Mbar_{g,n}$, but their construction is based on stacks of stable maps rather than a universal property as in the present preprint. We hope to understand the relation between these approaches more fully in future. 

An extension of the Abel-Jacobi map over a large locus in $\Mbar_{g,n}$ (when $k=0$) was produced some time ago by Bashar Dudin \cite{Dudin2015Compactified-un}. His locus depended on a choice or ramification data, as in our construction. But he did not make modifications of $\Mbar_{g,n}$, and so was not able to extend over the whole of the boundary. 

\subsection*{Acknowledgements}

I am very grateful to Alessandro Chiodo for many extensive discussions and much encouragement, without which this paper would not have been written. I have also benefitted greatly while writing this paper from discussions with Robin de Jong, Bas Edixhoven, Rahul Pandharipande, David Rydh, Johannes Schmitt, Arne Smeets, Nicolo Sibilla, Mattia Talpo, Jonathan Wise, Paolo Rossi, and many others. I am also grateful to Dimitri Zvonkine for encouraging me to include the extension to the twisted case, and to Nicola Pagani for pointing out that the construction works in the Chow ring with integral (rather than just rational) coefficients. Finally, I would like to thank the anonymous referee for a very helpful report.

\section{Notation and setup}

\subsection*{Base ring}

We work over the fixed base scheme $\Lambda \coloneqq \on{Spec}\bb Z$ equipped with the trivial log structure. The reader who prefers to take $\Lambda = \on{Spec} \bb C$ can freely do so with no modifications to what follows, and a substantial simplification to the proofs of \ref{thm:properness_lozenge} and \ref{lem:existence_of_cpct}. All our constructions commute with arbitrary base-change over $\Lambda$.

%In this paper we work relative to a fixed base scheme which we will denote $\Lambda$, and equip with the trivial log structure. 
%
%. Probably many readers will prefer to take $\Lambda = \on{Spec} \bb C$; almost all of the interesting features are already present in this case. But for the sake of readers interested in greater generality, we discuss here what assumptions on $\Lambda$ are needed in which sections of the paper. When we need to view $\Lambda$ as a log scheme (or stack), we will endow it with the trivial log structure. 
%
%We work with stacks $\bar{ \ca M}_{g,n}$, and variants. These objects make sense for any $\Lambda$, and our constructions can be carried out in the same generality. The universal property we prove in \ref{sec:universal_property_md} will need the existence of some resolutions of singularities. This is automatic in characteristic zero, but by restricting to log smooth test objects it is enough to assume that $\Lambda$ is regular. %In \ref{sec:GHRR} we use intersection theory, for which more assumptions are needed; it is safe to take the base to be a dedekind scheme (for example, a point). 
%
%Note that we can always safely take $\Lambda = \on{Spec} \bb Z$. Moreover, formation of our moduli spaces will commute with base change over $\Lambda$, so if we carry out the constructions over $\bb Z$ we can recover them for any other base by base-change. 
\begin{remark}
It seems that the definition of the double ramification cycle given in \cite{Graber2005Relative-virtua} can readily be extended over $\bb Z$, though we have not verified this. The computation of \cite{Janda2016Double-ramifica} is carried out over $\bb C$, and gives the class of the double ramification cycle as the value at zero of a polynomial $P$ in a variable $r$. The value of $P$ at a given value of $r$ is computed using $r$-th roots of line bundles (cf. Chiodo's formulae \cite{Chiodo2008Towards-an-enum}), which may give problems in characteristic dividing $r$. But in fixed characteristic $\ell$, the polynomial $P$ is completely determined by its values on integers coprime to $\ell$, so this problem can be circumvented (the author is grateful to Felix Janda for pointing this out). So it remains likely that the results of \cite{Janda2016Double-ramifica} can be extended to arbitrary characteristic. 
\end{remark}

%, and it is not clear how to extend over $\bb Z$;  they make use of spaces of $r$th roots of line bundles (and Chiodo's formulae \cite{Chiodo2008Towards-an-enum} for pushforwards of certain classes from them), which seems to give trouble in characteristic dividing $r$. Thus for any fixed ramification data, the formula of \cite{Janda2016Double-ramifica} will be valid in sufficiently large prime characteristic. The definitions in this paper work in any characteristic, and the same holds for our formulae in genus 1, but it remains to be seen how things work out in higher genus. 

\subsection*{Stack of weighted stable curves}%\label{sec:stack_of_pointed_curves}
For us, `curve' means proper, flat, finitely presented, with reduced connected nodal geometric fibres. Rather than treating each $\Mbar_{g,n}$ and weighting $a_1, \dots, a_n$ separately, we denote by $\Mbar$ the stack of stable pointed curves together with a $k$-twisted integer weighting --- in other words, points of $\Mbar$ consist of tuples 
\begin{equation*}
(C, x_1, \dots, x_n, a_1, \dots, a_n, k)
\end{equation*}
where the $x_i$ are the marked sections of our stable curve, and the $a_i$ and $k$ are integers satisfying
\begin{equation*}
\sum_i a_i  = k(2g-2). 
\end{equation*}
This stack is smooth over $\Lambda$, but is far from being connected --- it is a countably infinite disjoint union of substacks, each proper over $\Lambda$. 

%Rather than treating each $\ca M_{g,n}$ and weighting $a_1, \dots, a_n$ separately, we consider the stack of all `pointed stable curves together with a neutral integer weighting', which we now define. 

%
%', whose objects over a scheme $T/\Lambda$ consist of:
%\begin{itemize}
%\item
%A Zariski-locally constant function $n\colon T \to \bb Z_{\ge 0}$ (write $\sqcup_n T \to T$ for the corresponding disjoint union of copies of $T$, locally);
%\item A curve $C/T$;
%\item A section $\sqcup_u T \to C$; 
%\item A locally constant function $d\colon \sqcup_n T \to \bb Z$, 
%\end{itemize}
%such that for each geometric point $p \in T$, we have
%\begin{itemize}
%\item The fibre $C_p$ together with the multisection $\sqcup_n T_p \to C_p$ is a stable marked curve;
%\item $\sum_{q \in \sqcup_n T_p} d(q) = 0$.
%\end{itemize}

%\begin{remark}
%If we later decide we don't need stability, can just remove it here and add some `SNC singularities' assumption in its place. 
%\end{remark}

We write $\ca C/\Mbar$ for the universal curve, and $\ca J = \on{Pic}^0_{\ca C/\Mbar}$ for the universal jacobian (a semiabelian scheme, the fibrewise connected component of the identity in $ \on{Pic}_{\ca C/\Mbar}$). Let $\ca M$ denote the open substack of $\Mbar$ parametrising smooth curves. We write $x_i$ for the tautological sections, and $\Sigma$ for the Cartier divisor on $\ca C$ given by $\sum_i a_i x_i$. Then $\sigma \in \ca J_{\ca M}(\ca M)$ denotes the tautological section given by $\omega^{\otimes k}_{\ca C}(-\Sigma)$.

\subsection*{Log structures}

We work with log structures in the sense of Fontaine-Illusie, using Olsson's generalisation to stacks \cite{Olsson2001Log-algebraic-s}. We put log structures on $\ca C/\Mbar$ following Kato \cite{Kato1996Log-smooth-defo}, and the log structure on $\Mbar$ will be denoted $\alpha_{\Mbar}\colon P_{\Mbar} \to \ca O_{\Mbar}$, etc. 

%We work with log structures in the sense of Fontaine-Illusie, using Olsson's generalisation to stacks \cite{Olsson2001Log-algebraic-s}. Noting that $\ca C/\Mbar$ is naturally isomorphic to a disjoint union of $\Mbar_{g,n}$s together with their universal curves, we put log structures on these objects following Kato \cite{Kato1996Log-smooth-defo}. We denote these log stacks by $\ca C/\Mbar$. The log structure on $\Mbar$ will be denoted $\alpha_{\Mbar}\colon P_{\Mbar} \to \ca O_{\Mbar}$, etc. 

%If we want to talk about the underlying (non-logarithmic) algebraic stacks we will denote this in the standard way by underlining. %either say it in words, or perhaps use the `underline' notation if needed. 

If $P$ is a (sheaf of) monoid(s), we write $\bar P \coloneqq P / P^\times$; this notation does not sit well with the notation $\Mbar$ for the moduli stack of stable curves, but both are very standard, and there is no actual ambiguity.

%We use script letters for the logarithmic versions of these objects, eg. $\cl C/\cl M$. In particular, $\cl M$ is just $\ca M$ together with the divisorial log structure coming from the boundary of $\mathring{\ca M}$. {I'm not convinced loading the scripts like this is actually a good idea. Just saw we always work with the log objects, and if we want to talk about the underlying algebraic stacks we do so. }

\subsection*{Weightings on a graph}%\label{sec:weightings}
A graph consists of a finite set $V$ of vertices, a finite set $H$ of half-edges, a map `$\on{end}$' from the half-edges to the vertices, an involution $i$ on the half-edges, a \emph{genus} $g\colon V \to \bb Z_{\ge 0}$, and a \emph{twist} $k\in \bb Z$. Graphs are assumed connected unless stated otherwise, and the genus of a graph is its first Betti number plus the sum of the genera of the vertices. 

\emph{Self-loops} are when two distinct half-edges have the same associated vertex and are swapped by $i$. We define \emph{edges} as sets $\{h, h'\}$ (of cardinality 2) with $i(h) = h'$. \emph{Legs} are fixed points of $i$. A \emph{directed edge} $h$ is a half-edge that is not a leg; we call $\on{end}(h)$ its \emph{source} and $\on{end}(i(h))$ its \emph{target}, and sometimes write it as $h\colon \on{end}(h) \to \on{end}(i(h))$. We write $E = E(\Gamma)$ for the set of edges, and $\stackrel{\to}{E}$ for the set of directed edges. %, but we mostly use the latter as it is more common in the literature. 

The \emph{valence} $\on{val}(v)$ of a vertex is the number of non-leg half-edges incident to it, and we define the \emph{canonical degree} $\kappa(v) = 2g(v) - 2 + \on{val}(v)$, so that 
\begin{equation*}
2g(\Gamma) - 2 = \sum_v \kappa(v). 
\end{equation*}

%\emph{Directed edges} are ordered pairs $(h,h')$ with $i(h) = h'$; we call $\on{end}(h)$ the \emph{source} and $\on{end}(h')$ the \emph{target}.
%There are two obvious correspondences between non-leg half-edges and directed edges; we think of a half edge at a vertex $v$ as an edge directed \emph{out of} $v$. 

 A \emph{closed walk} in $\Gamma$ is a sequence of directed edges so that the target of one is the source of the next, and which begins and ends at the same vertex. We call it a \emph{cycle} if it does not repeat any vertices or (undirected) edges.

\begin{definition}\label{def:weighting}
A \emph{$G$-weighting} is a function $w$ from the half-edges to a group $G$ such that:
\begin{enumerate}
\item If $i(h) = h'$ and $h \neq h'$ then $w(h) +w(h') =0$;
\item For all vertices $v$, $\sum_{\on{end}(h) = v} w(h) +  k\kappa(v)$ = 0. 
\end{enumerate}
\end{definition}
When the twist $k = 0$, a $G$-weighting can be though of as a flow of an incompressible fluid around the graph. 

A $G$-leg-weighted graph is a graph together with a function from the legs to $G$, such that the sum over all the legs is $-k(2g(\Gamma) - 2)$. If $\Gamma$ is a $G$-leg-weighted graph we write $W(\Gamma)$ for the set of weights on $\Gamma$ which restrict to the given values on the legs. It is easy to see that $W(\Gamma)$ is never empty (this uses the running assumption of connectedness). After choosing an oriented basis of $\on H_1(\Gamma, \bb Z)$, the set $W(\Gamma)$ becomes a torsor under $\on H_1(\Gamma, G)$. In this article we will use only the cases $G = \bb Z$ and $G = \bb Q$. % (and not $G = \bb Z/r\bb Z$, which plays a large role in the work of \cite{Janda2016Double-ramifica}). 
We will refer to weightings taking values in $\bb Z$ just as \emph{weightings}. %, and weightings taking values in $\bb Q$ as \emph{$\bb Q$-weightings}. 

\subsection*{Combinatorial charts}%\label{sec:combinatorial_charts}

If $p \colon \on{Spec} k \to \Mbar$ is a geometric point, we have an associated graph $\Gamma_p$. If a node of the curve over $p$ has local equation $xy - r$ for some $r \in \ca O^{et}_{\Mbar, p}$, then the image of $r$ in the monoid $\bar P_{\Mbar, p}$ is independent of the choice of local equation. In this way we define a map $\ell\colon E(\Gamma_{p}) \to \bar P_{\Mbar, p}$, recalling that edges of the graph correspond to nodes of the curve. This is a logarithmic version of the labelling defined in \cite{Holmes2014Neron-models-an}. 

Given a leg-weighted graph $\Gamma$ with edge set $E$, set $\Mbar_\Gamma = \on{Spec}\Lambda[\bb N^E]$, equipped with the toric divisorial log structure. As usual we write $\alpha\colon P_{\Mbar_\Gamma} \to \ca O_{\Mbar_\Gamma}$ for the map from the sheaf of monoids to the structure sheaf. 

To any point $p$ in $\Mbar_\Gamma$ we associate the graph $\Gamma_p$ obtained from $\Gamma$ by contracting exactly those edges $e$ such that the corresponding basis elements of $\bb N^E$ specialise to units at $p$. We define a map $\ell\colon E(\Gamma_{p}) \to \bar P_{\Mbar_\Gamma, p}$ by sending an edge $e$ to the image of the associated basis element of $\bb N^E$. The map naturally lifts to $P_{\Mbar_\Gamma, p}$, and does not send any edge to a unit, by definition of $\Gamma_p$. 

A \emph{combinatorial chart of $\Mbar$} consists of a leg-weighted graph $\Gamma$ and a diagram of log stacks
\begin{equation*}
\Mbar \stackrel{f}{\longleftarrow} U \stackrel{g}{\longrightarrow} \Mbar_\Gamma
\end{equation*}
%\begin{center}
%\begin{tikzcd}
%U \arrow[r, "g"] \arrow[d, "f"] & \Mbar_\Gamma\\
%\Mbar & \\
%\end{tikzcd} 
%\end{center}
satisfying the following five conditions:
\begin{enumerate}
\item $U$ is a connected log scheme
\item $g\colon U \to \Mbar_\Gamma$ is strict and log smooth
\item $f\colon U \to \Mbar$ is strict and log \'etale
\item the image of $g$ meets the minimal stratum of $\Mbar_\Gamma$. 
\end{enumerate}
Let $p\colon \on{Spec}k \to U$ be any geometric point, yielding natural maps 
\begin{equation*}
\bar P_{\Mbar, f \circ p} \stackrel{f^\flat}{\to} \bar P_{U, p} \stackrel{g^\flat}{\leftarrow} \bar P_{\Mbar_\Gamma, g \circ p}. 
\end{equation*}
\begin{enumerate}[resume*]
\item 
We require the existence of an isomorphism 
\begin{equation*}
\phi_p\colon \Gamma_{f \circ p} \to \Gamma_{g \circ p}
\end{equation*}
such that
$f^\flat(\ell(e)) = g^\flat (\ell(\phi_p(e))$ for every edge $e$ (which necessarily makes this $\phi_p$ unique if it exists). Moreover, the map $\phi_p$ sends the leg-weighting on $\Gamma_{f \circ p}$ coming from the $-a_i$ to the leg-weighting on $\Gamma_{g \circ p}$ coming from that on $\Gamma$. 
%\item For $p$ as above mapping to the closed stratum of $\Mbar_\Gamma$, the induced map $\Gamma_{f \circ p} \to \Gamma_{g \circ p} = \Gamma$ preserves the leg-weighting sends the leg-weighting coming from the $-a_i$ to the given one on $\Gamma$. 
\end{enumerate}

%\begin{itemize}
%\item
%a connected log scheme $U$
%\item a graph $\Gamma$ with a neutral integer weighting of its legs (cf. \ref{sec:weightings})
%\item a strict log smooth map $U \to\Mbar_\Gamma$ whose image contains the closed stratum $m$ of $\Mbar_\Gamma$
%\item a strict log \'etale map $U \to \Mbar$ such that the restriction of the pulled-back curve to the fibre over $m$ has constant topological type
%\item an isomorphism of leg-weighted graphs between $\Gamma$ and the graph of the pulled-back curve over the fibre over $m$. 
%\item For each edge of $\Gamma$, we require that the labels on $\Mbar$ and on $\Mbar_\Gamma$ pull back to the same thing on $U$. 
%\end{itemize}

%Given a geometric point of $u$ we can associate two graphs, one from $\Mbar$, and the other from $\Mbar_\Gamma$. The isomorphism of graphs given above will induce an isomorphism between these two (leg-weighted) graphs. 

It is clear from \cite{Kato2000Log-smooth-defo} that $\Mbar$ can be covered by combinatorial charts. We will first construct $\Md_\Gamma$ over the $\Mbar_\Gamma$, and then descend it to $\Mbar$. 

%\begin{remark}
%If $e$ is an edge of $\Gamma$, write $\ell(e) \in \bb N^E$ for the corresponding basis vector. Then $\alpha(\ell(e))$ is the label of the edge $e$ in the sense of \cite{Holmes2014Neron-models-an}. 
%\end{remark}

\section{Construction of $\Md$}\label{sec:construction}

\subsection*{The cone $c_w$ associated to a weighting}%\label{sec:cone_from_weighting}

For the remainder of this subsection we fix a combinatorial chart with leg-weighted graph $\Gamma$, writing $E$ for the edge set. To a weighting $w \in W(\Gamma)$ we will associate a rational polyhedral cone $c_w$ inside the positive orthant of $\bb Z^E \otimes_{\bb Z} \bb Q$. Recalling that $\Mbar_\Gamma  = \on{Spec} \Lambda[\bb N^E]$, such a cone will induce an affine toric scheme over $\Mbar_\Gamma$ in the usual way, cf. \cite{Fulton1993Introduction-to}. Such an object has  a natural log structure. We will build $\Md_\Gamma$ by glueing together affine patches of this form. %, and $\MD_\Gamma$ by a slight generalisation using small sets of weightings:
%\begin{definition}
%If $\Gamma$ is a leg-weighted graph and $W \sub W(\Gamma)$ is a subset of the weightings, we say $W$ is \emph{small} if it is non-empty and, for every edge $e$ of $\Gamma$ and pair of weightings $w$, $w' \in W$, we have 
%\begin{equation*}
%\abs{w(e) - w'(e)} \le 1. 
%\end{equation*}
%\end{definition}
%\noindent Small sets are always finite, and singletons are always small.  

Fix a weighting $w \in W(\Gamma)$ and let $\gamma$ be an oriented cycle in $\Gamma$. If $e$ is a directed edge appearing in $\gamma$, we define $w_\gamma(e)$ to be the value of $w$ on the first half-edge of $e$ --- we might think of this as the flow along $e$ in the direction given by $\gamma$. 

\begin{definition}\label{eq:c_w}
Let $t \in \bb Q_{\ge 0}^E$; we refer to such an element as a \emph{thickness}. We say $t$ is \emph{compatible with $w$} if for every cycle $\gamma$ we have 
\begin{equation}\label{eq:compat}
\sum_{e \in \gamma} w_\gamma(e) t(e) = 0. 
\end{equation}
\end{definition}
%For this fixed $w$, we see easily that sums and scalar multiplies of compatible thicknesses are again compatible. It is enough to check the compatibility condition for $\gamma$ running over an oriented basis for the homology, so there are in effect only finitely many conditions, each given by an equation with rational coefficients (since the $w_\gamma(e)$ are rational, even integral). It is then clear that 
One checks easily that the set of all thicknesses $t$ which are compatible with a given weighting $w$ form a rational polyhedral cone in $\bb Q_{\ge 0}^E$, which we denote by $c_w$. 

\begin{lemma}\label{rem:common_nonzero_thickness_implies_equal}
Suppose $w, w' \in W(\Gamma)$ and that $c_w \cap c_{w'}$ contains a thickness $t$ which does not vanish on any edge. Then $w = w'$. 
\end{lemma}
\begin{proof}
Writing $\tilde w = w - w'$, we see that $\tilde w$ is a `weighting' for the graph $\Gamma$ with all the leg-decorations removed; we can think of this $\tilde w$ as a flow of an incompressible fluid around $\Gamma$, with no sources and sinks. Suppose that $\tilde w$ is not everywhere zero. We then claim that there exists a directed cycle $\gamma$ in $\Gamma$ such that, for every directed edge $e \in \gamma$, we have $\tilde w_\gamma(e) >0$. To build such a cycle, we begin on any directed edge with positive flow $\tilde w$. The incompressibility condition then implies that we can continue the path along another edge, still having $\tilde w >0$. Continuing in this way, the finiteness of the graph forces this path to intersect itself at some point. Possible discarding the beginning of this path, we have the desired cycle $\gamma$. Then 
\begin{equation*}
\sum_{e\in \gamma} \tilde w_\gamma(e)t(e) = 0, 
\end{equation*}
but all the $t(e)>0$, and all the $\tilde w_\gamma(e) >0$, a contradiction. 
\end{proof}

% If $W\sub W(\Gamma)$ is a small set, we define $c_W$ to be the convex hull of the $c_w$ for $w \in W$ (in particular, $c_{\{ w\}} = c_w$).

\begin{definition}\label{def:fans}
We write $F_\Gamma$ for the set of faces of the cones $c_w$ as $w$ runs over weightings in $W(\Gamma)$. %We write $\bar F_\Gamma$ for the set of faces of the cones $c_W$ as $W$ runs over small sets of weightings. 
\end{definition}

\begin{remark}[Example: 2-gon, $k=0$]
Suppose the graph $\Gamma$ has two edges and two (non-loop) vertices $u$ and $v$. Suppose the leg weighting is $+n$ at $u$ and $-n$ at $v$. Weightings consist of a flow of $a$ along one edge from $u$ to $v$, and $n-a$ along the other (again from $u$ to $v$), for $a \in \bb Z$. The cone $c_w$ is non-zero if and only if both $a$ and $n-a$ are non-negative, and for such $a$ the cone $c_w$ is the ray in $\bb Q_{\ge 0}^2$ through the point $(n-a, a)$. Thus we get exactly $n+1$ rays in the positive quadrant. This is the fan $F_\Gamma$. %There is an obvious way to `fill it in' to a complete fan (indeed, exactly one way to do so without introducing new rays), and this is exactly $\bar F_\Gamma$. 
\end{remark}

\begin{remark}[Example: 3-edge banana, $k=0$]
Suppose again that the graph has two vertices $u$ and $v$, but now three edges between them. Suppose that the weighting is $+10$ on $u$ and $-10$ on $v$. \Cref{fig:fan_Md} shows the slice through the incomplete fan $F_\Gamma \sub \bb Q_{\ge 0}^3$ where the sum of the values of the thickness on the edges is 1. %A smooth approximation of a slice of the complete fan $\bar F_\Gamma \sub \bb Q_{\ge 0}^3$ is shown in \ref{fig:fan_MD}. 
\end{remark}

In the next subsection we will verify that $F_\Gamma$ is a \emph{finite} fan (in the sense of toric geometry). The reader might prefer to skip the details, as they play little role in what follows. 

\begin{figure}
\centerline{\includegraphics[scale = 0.5]{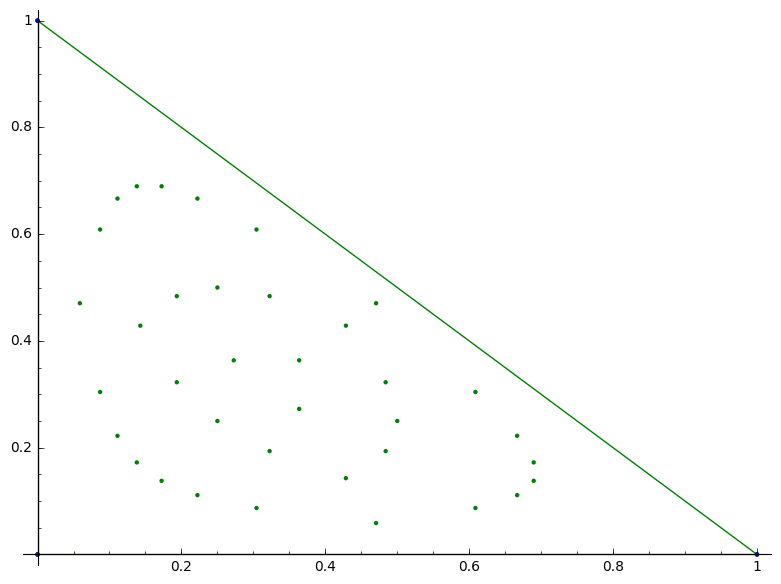}}
\caption{Slice of $F_\Gamma$}\label{fig:fan_Md}
\end{figure}

%\begin{figure} 
%\includegraphics[width=\textwidth]{Mhatcomp}
%\caption{Fan of $\MD$}
%\label{fig:fan_MD}
%\end{figure}
% These flaws are an artefact of the way the picture was created: by plotting those thicknesses for which `the' induced weighting takes integral values on at least one edge (the thicknesses in $F_\Gamma$ are roughly those such that the resulting weighting takes integral value on every edge). The machinery of `small sets' of weightings is designed to formalise all this. 

\subsection*{$F_\Gamma$ is a finite fan}
%\begin{remark}
%{I have an idea that may simplify some of the results in this section (though probably not the proof of finiteness, which is the longest), see text just below.} Choose an orientation on each edge of the graph. Then each cycle corresponds to an orthant in $\bb Z^E$, and we write $P$ for the intersection of these orthants. 
%
%Then weightings are naturally points in $\bb Z^E$, and there is a natural extension to rays in $\b Q^E$. If $w$ is a weighting, then inner product with $w$ induces a map $\phi_w\colon \bb Q_{\ge 0}^E \to \bb Q^E$, and $c_w$ the the preimage of $P$ under this map. 
%
%I think the (rays spanned by) the $w$ forms a fan in $\bb Q^E$ (not finite). Does the `inner product, the pull back $P$' map induce a new fan for formal reasons? 
%
%\end{remark}

%In this section we will verify that  $F_\Gamma$ is a finite fan. The reader who is prepared to take this result on trust can safely skip this section. 

\begin{lemma}\label{lem:cones_intersect_in_face}
Let $w_1$, $w_2 \in W(\Gamma)$ be two weightings. Then the intersection of the cones $c_{w_1}$ and $c_{w_2}$ is a face of $c_{w_1}$. 
\end{lemma}
\begin{proof}
Let $t \in c_{w_1} \cap c_{w_2}$. For an edge $e$ with $t(e) \neq 0$ we claim $w_1(e) = w_2(e)$. To see this, let $\Gamma_t$ be the graph obtained from $\Gamma$ by contracting those edges on which $t$ vanishes. Then each $w_i|_{\Gamma_t}$ is a weighting compatible with $t|_{\Gamma_t}$, and the latter does not vanish on any edge, so $w_1|_{\Gamma_t} = w_2|_{\Gamma_t}$ by \ref{rem:common_nonzero_thickness_implies_equal}. 

Define
\begin{equation*}
E_= = \{ e \in E(\Gamma) : w_1(e) = w_2(e)\}, \;\;\;\; E_{\neq} = E(\Gamma) \setminus E_=. 
\end{equation*}
Writing $\Gamma_=$ for the graph obtained from $\Gamma$ by contracting $E_{\neq}$, and set $w_= \coloneqq w_1|_{\Gamma_=} = w_2|_{\Gamma_=}$. Define $c_{=}$ to be the cone in $\bb Q_{\ge 0}^{E_=}$ corresponding to the weighting $w_=$. 
By the claim above we see that every $t \in c_w \cap c_{w'}$ vanishes on every edge of ${E_{\neq}}$. 

%Hence we see that, for every edge $e \in E(\Gamma)$, either $w_1(e) = w_2(e)$, or $t(e) = 0$ for every thickness $t \in c_{w_1} \cap c_{w_2}$. 

I now claim that $c_{w_1} \cap c_{w_2} = c_= \times \{ \mathbf{0}\}$, where $\mathbf{0}$ is the zero vector in $\bb Q_{\ge 0} ^{E_{\neq}}$. Well, it is clear that $c_{w_1} \cap c_{w_2} \subseteq c_= \times \{ \mathbf{0}\}$ by definition of $E_=$ and $E_{\neq}$. For the other inclusion, let $t \in c_= \times \{ \mathbf{0}\}$, and let $\gamma$ be a cycle in $\Gamma$. Write $\gamma_=$ for the closed walk in $\Gamma_=$ resulting from contracting $\gamma$. By definition of $c_=$ we have that 
\begin{equation*}
\sum_{e \in \gamma'} w_=(e)t(e) = 0, 
\end{equation*}
and $t$ vanishes on $E_{\neq}$, so we see 
\begin{equation*}
\sum_{e \in \gamma} w_1(e)t(e) = 0 = \sum_{e \in \gamma} w_2(e)t(e) = 0
\end{equation*}
 as required. 
\end{proof}

Now we know that any two $c_w$ intersect in a face. The following well-known lemma shows that $F_\Gamma$ is a fan. 
%The next result then shows that $F_\Gamma$ is a fan; surely it is well-known, but we could not find it in the literature, so we provide a proof:
\begin{lemma}\label{lem:intersect_in_faces_gives_fan}%{Notation in this proof is  a bit clumsy. replace use of $C, C'$ by $c_1$, $ c_2$? }
Let $\Phi_0$ be a set of cones in $\bb Q^n$, and assume that for all $C$, $C' \in \Phi_0$, the intersection $C \cap C'$ is a face of $C$. Let $\Phi$ be the set of all faces of cones in $\Phi_0$. Then $\Phi$ is a fan. 
\end{lemma}
%\begin{proof}
%This follows easily from standard facts about faces in \cite[section 1.2]{Fulton1993Introduction-to}, and we omit the details. This result is surely standard, but we could not find a reference. 
%\expanded{We need to show that every intersection of faces of cones in $\Phi_0$ is a face of a cone in $\Phi_0$. So let $C$, $C'$ be two cones, and $F$, $F'$ faces of $C$, $C'$ respectively. We want to show that $F \cap F'$ is a face of $C'$. We recall basic properties of the faces of a cone:
%\begin{enumerate}
%\item A face of a cone is a cone. 
%\item
%A face of a face is a face (this makes sense by (1)). 
%\item An intersection of faces of a cone $c_0$ is a face of $c_0$. 
%\item If $G_1 \sub G_2$ are faces of a cone $c_0$ then $G_1$ is a face of $G_2$. 
%\end{enumerate}
%The first three are in \cite[section 1.2]{Fulton1993Introduction-to}, and the fourth is an easy exercise from the definition of a face in [loc.cit.]. 
%
%Now, I claim that if $G$ is a face of $C$ and is contained in $C'$ then it is a face of $C'$. Indeed, $G$ is contained in $C \cap C'$, and so by (4) is a face of $C \cap C'$, so by (2) it is a face of $C'$. 
%
%Now $F \cap C' = F \cap (C \cap C')$ which is the intersection of $F$ with a face of $C$ by assumption. Hence by (3) we know that $F \cap C'$ is a face of $C$. Then by the claim just above we have that $F \cap C'$ is a face of $C'$. So $F \cap F' = (F \cap C) \cap F'$ is a face of $C'$ by (3). }
%\end{proof}

\begin{lemma}\label{lem:finitely_many_cones}
The \emph{set} of cones $\{ c_w: w \in W(\Gamma)\}$ is finite. 
\end{lemma}
\begin{proof}
The proof is by induction on $\on{h}^1(\Gamma)$. Recalling that $W(\Gamma)$ is a torsor under $\on{H}_1(\Gamma, \bb Z)$, we see that $W$ is finite whenever $\on{h}^1(\Gamma) = 0$, so in this case there is nothing to prove. In general, we say that a weighting $w$ on $\Gamma$ \emph{admits a positive cycle} if there exists a cycle $\gamma$ in $\Gamma$ such that $w(e) >0$ for every $e \in \gamma$. In the next two lemmas, we will show 
\begin{enumerate}
\item For fixed $\Gamma$, all but finitely many weightings admit a positive cycle (\ref{lem:existence_of_positive_cycle}). 
\item If $\gamma$ is a positive cycle for $w$ and $\Gamma/\gamma $ is the graph obtained from $\Gamma$ by contracting every edge in $\gamma$, then 
\begin{equation*}
c_w = c_{w|_{\Gamma/\gamma}} \times \{ \mathbf{0}\}, 
\end{equation*}
where $c_{w|_{\Gamma/\gamma}} \sub \bb Q_{\ge 0}^{E(\Gamma/\gamma)}$ is the cone associated to the restricted weighting $w|_{\Gamma/\gamma}$, and $\mathbf{0}$ is the zero vector in $\bb Q_{\ge 0}^{E(\gamma)}$ (\ref{lem:decomposition_of_weighting}). 
\end{enumerate}
Now there are only finitely many cycles in $\Gamma$, and for every cycle $\gamma$ we have that $\on{h}^1(\Gamma/\gamma) < \on{h}^1(\Gamma)$, hence by our induction hypothesis we have that there are only finitely many cones for $\Gamma/\gamma$. Putting these ingredients together concludes the proof. %: up to ignoring finitely many $w$ (those not admitting a positive cycle), we can assume our cone $c_w$ is of the form $c_{w|_{\Gamma/\gamma}} \times \{ \mathbf{0}\}$ where $\gamma$ runs over the finite set of cycles in $\Gamma$, and $c_{w|_{\Gamma/\gamma}}$ runs over set of cones for $\Gamma/\gamma$, which is finite by our induction hypothesis. 
\end{proof}

\begin{lemma}\label{lem:existence_of_positive_cycle}
For fixed $\Gamma$, all but finitely many $w$ admit a positive cycle
\end{lemma}
\begin{proof}%{This is a bit of a mess with directed edges - would it be clearer to revise and everything in terms of half-edges?}
\noindent \textbf{Step 1:} setup. 

Fix a weighting $w$, and fix a basis $B$ of $\on{H}_1(\Gamma, \bb Z)$ consisting of cycles. We can think of $b \in B$ as a function from the set $\stackrel{\to}{E}$ of directed edges of $\Gamma$, sending a directed edge $e$ to 0 is $e \notin b$, and $\pm1$ otherwise (depending on whether the orientation of $e$ agrees with $b$). Given an element $\mathbf v \in \bb Z^B$, we define $\on{cycle}(v)$ to be the function $\stackrel{\to}{E} \to \bb Z$ given by $\sum_b \mathbf{v}_b b$. This is somewhat clumsy notation, as it would be nicer just to think of $\mathbf v$ as an element in the cycle space, but distinguishing carefully between $\mathbf v$ and $\on{cycle}(\mathbf v)$ seems important to avoid confusion in this proof. In this way we see that every weighting on $\Gamma$ is of the form $w + \on{cycle}(\mathbf{v})$ for a unique $\mathbf{v} \in \bb Z^B$. 

Define recursively a function $\phi\colon \bb Z_{\ge 0} \to \bb Z$ by setting $\phi(0) = 1$ and $\phi(n) = \sum_{0 \le j < n} \phi(j)$. Write $m \coloneqq \max_{e \in E(\Gamma)}\abs{w(e)}$, and $h \coloneqq \on{h}^1(\Gamma)$. Define $N = m\phi(h)$ (this is rigged exactly to make step 3 of this argument work). Define
\begin{equation*}
B(N) = \left\{\mathbf{v} \in \bb Z^B : \max_{b \in B} \abs{\mathbf{v}_b} \le N\right\}, 
\end{equation*}
a finite set. In the remainder of this argument, we will show that for every $\mathbf v \in \bb Z^B \setminus B(N)$, the weighting $w + \on{cycle}(\mathbf v)$ admits a positive cycle. To this end, we fix for the remainder of the argument a $\mathbf v \in \bb Z^B \setminus B(N)$. 

\noindent \textbf{Step 2:} ordering the $\mathbf v_{b}$. 

Up to changing the orientations of the elements of $B$, we may assume that all the integers $\mathbf{v}_b$ are non-negative. Put an ordering on $B$ so that the $\mathbf v_{b_i}$ are in increasing order: 
\begin{equation*}
0 \le \mathbf v_{b_1} \le \mathbf v_{b_2} \le \cdots \le \mathbf v_{b_h}, 
\end{equation*}
with $\mathbf v_{b_h} > N$. 

\noindent \textbf{Step 3:} choosing a critical $b_{r}$. 

We now show by a small computation that there exists $1 \le r \le h$ such that 
\begin{equation*}
m + \sum_{1 \le i \le  r-1} \mathbf v_{b_i} < \mathbf v_{b_r}. 
\end{equation*}
Indeed, suppose no such $r$ exists. Then for each $1 \le j \le h$ we have 
\begin{equation*}
m + \sum_{1 \le i \le  j-1} \mathbf v_{b_i} \ge \mathbf v_{b_j}, 
\end{equation*}
and induction on $j$ yields $\mathbf v_{b_j} \le m \phi(j)$ for all $1 \le j \le h$, contradicting our assumption that $\mathbf v_{b_h} > N = m\phi(h)$. 

From now on, we fix such an $r$. 

\noindent \textbf{Step 4:} finding the positive cycle $\gamma$. 

%For this step, we ignore $w$ and the $\mathbf v_{b_i}$ for $i < r$ (though obviously these will be important later when we come to proving that this $\gamma$ is indeed positive). 

Define a function $f\colon \stackrel{\to}{E} \to \bb Z$ sending a directed edge $e$ to 
\begin{equation*}
\sum_{r \le j \le h} \left\{ \begin{array}{cl}
1 & \text{if }e \in b_j\text{ and has same direction as }b_j\\
-1& \text{if }e \in b_j\text{ and has opposite direction to }b_j\\ 
0& \text{ otherwise (i.e. }e \notin b_j\text{)}\\
\end{array} \right\}
\end{equation*}
In the remainder of this step we will show that there exists a cycle $\gamma$ with $f(e) > 0$ for every directed edge $e \in \gamma$. In step 5 we will see that any such $\gamma$ is necessarily a positive cycle. 

First, because the $b_j$ are part of a basis, we see that $f$ is not identical to zero. Hence there is a directed edge $e$ with $f(e) > 0$. We build a path in $\Gamma$ starting with $e$ by the following procedure: whenever we hit a vertex $v$, choose an edge $e_v$ out of $v$ such that $f(e_v) > 0$. Why is this always possible? Note that the sum over all edges $e''$ into $v$ of $f(e'')$ is necessarily zero, and since we arrived at $v$ along an edge with $f > 0$, there must also be an edge leaving $v$ with $f>0$. 

Since $\Gamma$ is finite, this path must eventually meet itself, say at a vertex $v_0$. Deleting the start of the path up to $v_0$ yields the cycle $\gamma$ that we sought. 

\noindent \textbf{Step 5:} showing that $\gamma$ is indeed a positive cycle for the weighting $w + \on{cycle}(\mathbf v)$. 

Choose any $\gamma$ as in step 4. Define functions $F$, $G\colon \stackrel{\to}{E}\to\bb Z$ by the following formulae: 
\begin{equation*}
F(e) = \sum_{r \le j \le h} \left\{ \begin{array}{cl}
\mathbf v_{b_j} & \text{if }e \in b_j\text{ and has same direction as }b_j\\
-\mathbf v_{b_j}& \text{if }e \in b_j\text{ and has opposite direction to }b_j\\ 
0& \text{ otherwise (i.e. }e \notin b_j\text{)}\\
\end{array} \right\}
\end{equation*}
\begin{equation*}
G(e) = w(e) + \sum_{1 \le  j \le r-1} \left\{ \begin{array}{cl}
\mathbf v_{b_j} & \text{if }e \in b_j\text{ and has same direction as }b_j\\
-\mathbf v_{b_j}& \text{if }e \in b_j\text{ and has opposite direction to }b_j\\ 
0& \text{ otherwise (i.e. }e \notin b_j\text{)}\\
\end{array} \right\}. 
\end{equation*}
Observe that $F + G = w + \on{cycle}(\mathbf v)$ as functions $\stackrel{\to}{E} \to \bb Z$. Let $e \in \gamma$ be a directed edge; we will show that $F(e) + G(e) >0$ as required. 

Define $\epsilon = m + \sum_{1 \le i \le r-1} \mathbf v_{b_i}$ (cf. step 3). Note that $f(e) \ge 1$, and $F(e) > \epsilon f(e) \ge \epsilon$ since $\mathbf v_{b_j} > \epsilon$ for every $j \ge r$. Now 
\begin{equation*}
\abs{G(e)} \le m + \sum_{1 \le j \le r-1}\mathbf  v_{b_j} = \epsilon, 
\end{equation*}
hence
\begin{equation*}
(w + \on{cycle}(\mathbf v)) (e)  = F(e) + G(e) > \epsilon - \epsilon = 0. \qedhere
\end{equation*}
\end{proof}

\begin{lemma}\label{lem:decomposition_of_weighting}
Fix a weighting $w$. If $\gamma$ is a positive cycle for $w$ and $\Gamma/\gamma $ is the graph obtained from $\Gamma$ by contracting every edge in $\gamma$, then 
\begin{equation*}
c_w = c_{w|_{\Gamma/\gamma}} \times \{ \mathbf{0}\}, 
\end{equation*}
where $c_{w|_{\Gamma/\gamma}} \sub \bb Q_{\ge 0}^{E(\Gamma/\gamma)}$ is the cone associated to the restricted weighting $w|_{\Gamma/\gamma}$, and $\mathbf{0}$ is the zero vector in $\bb Q_{\ge 0}^{E(\gamma)}$.
\end{lemma}
\begin{proof}
The inclusion 
\begin{equation*}
c_w \supseteq c_{w|_{\Gamma/\gamma}} \times \{ \mathbf{0}\}
\end{equation*}
follows easily from the definition of the cone of a weighting, since every cycle in $\Gamma/\gamma$ arises by restricting some cycle in $\Gamma$. 

 We need to show the other inclusion, so let $t \in c_w$ be a thickness. The $t$ satisfies the equation
\begin{equation*}
\sum_{e \in \gamma} t(e) w(e) = 0, 
\end{equation*}
and since all the $w(e)$ are positive this forces all the $t(e)$ to vanish. %This immediately yields the other inclusion, and we are done. 
\end{proof}

Putting together lemmas \oref{lem:cones_intersect_in_face}, \oref{lem:intersect_in_faces_gives_fan}, and \oref{lem:finitely_many_cones} we immediately deduce:
\begin{corollary}\label{cor:form_fan}
The set of cones $F_\Gamma$ is a finite fan inside $\bb Q_{\ge 0}^{E}$. 
\end{corollary}

\subsection*{The construction of $\Md$}

\begin{definition}\label{rem:patches}
We define $\Md_\Gamma$ to be the toric scheme over $\Mbar_\Gamma$ defined by the fan $F_\Gamma$. %, and we define $\MD_\Gamma$ to be the toric scheme over $\Mbar_\Gamma$ defined by the fan $\bar F_\Gamma$. 
We equip it with the toric log structure. 

If $w\in W(\Gamma)$ we define $\Md_w$ to be the affine toric variety associated to $c_w$, an affine patch of $\Md_\Gamma$.  
\end{definition}

\begin{remark}
It follows from \cite{Fulton1993Introduction-to} that $\Md_\Gamma \to \Mbar_\Gamma$ is separated, of finite presentation, and normal. It is moreover logarithmically \'etale, since it is given by toric blowups. %, and that $\MD_\Gamma \to \Mbar_\Gamma$ is proper. 
\end{remark}

%\begin{definition}
%%, and if $W \sub W(\Gamma)$ is small we define $\MD_W$ to be the affine toric variety associated to $c_W$, again with toric log structures. These are affine patches of $\Md$, respectively $\MD$. 
%%The toric scheme $\Md_\Gamma$ is constructed by glueing together affine patches, one for each weighting $w \in W(\Gamma)$. The affine patch associated to $w$ is (by definition) the spectrum of the monoid ring of the dual monoid to the cone $c_w$; we denote it by $\Md_w$, since we will sometimes need to refer to it in what follows. Similarly, $\MD_\Gamma$ is made by glueing patches corresponding in the same way to small sets $W$; we will denote these patches $\MD_W$. 
%\end{definition}

Given a combinatorial chart $\Mbar \leftarrow U \to \Mbar_\Gamma$ we define $\Md_U$ by pulling back $\Md_\Gamma$ from $\Mbar_\Gamma$. Such $U$ form an \'etale cover of $\Mbar$, and the collection of $\Md_U$ is easily upgraded to a descent datum. 
\begin{definition}
We define $\pi^\loz\colon \Md \to \Mbar$ to be the algebraic space obtained by descending the $\Md_U$. 
\end{definition}

\expanded{Now we have constructed the $\Md_\Gamma$ and $\MD_\Gamma$, and we want to glue them to construct $\Md$ and $\MD$. We know $\Mbar$ is covered by combinatorial charts $U$, and over each $U$ we build $\Md_U$ (resp. $\MD_U$) by pulling back the relevant $\Md_\Gamma$ (resp. $\MD_\Gamma$). We need to upgrade this collection of $\Md_U$ (resp. $\MD_U$) to a descent datum, so that it will descend to give us a $\Md$ (resp. $\MD$). 

We will not give all the details of the construction of the descent datum, since it is rather similar to those considered in \cite{Holmes2014A-Neron-model-o} and \cite{Biesel2016Fine-compactifi}. However, the key step is to give a pullback map on the $\Md_U$. For the remainder of this paragraph we will discuss only $\Md$; the story for $\MD$ is exactly parallel. Suppose $U$ and $U'$ are combinatorial charts with graphs $\Gamma$, $\Gamma'$, and suppose that we have a map $U' \to U$ over $\Mbar$. We need to specify an isomorphism $\Md_{U'} \iso \Md_U \times_U U'$. Now $\Gamma'$ is naturally obtained from $\Gamma$ by contracting some edges, and a weighting on $\Gamma$ clearly induces a weighting on $\Gamma'$, giving a map $\on{res}\colon W(\Gamma) \to W(\Gamma')$ (which is a bijection if and only if the contraction does not kill any homology). If $w \in W(\Gamma)$ then one sees from the definition a natural isomorphism $\Md_w \times_{\Mbar_\Gamma} \Mbar_{\Gamma'} \iso \Md_{\on{res}(w)}$, and these glue to give the isomorphism $\Md_{U'} \iso \Md_U \times_U U'$ that we need. }

%We conclude this section with a theorem summarising some easy properties of $\Md$ and $\MD$. 
\begin{theorem}\label{thm:properties_of_Md_MD}
The stack $\Md$ is normal, and the map $\pi^\loz\colon \Md \to \Mbar$ is separated, of finite presentation, relatively representable by algebraic spaces, birational, and logarithmically \'etale. 
%\item $\pi^\lozenge$ and $\pi^\blacklozenge$ are . 
%\item $\pi^\blacklozenge$ is proper. 
%\item The section $\sigma$ (from \ref{sec:stack_of_pointed_curves}) extends uniquely to a section ${\sigma}^\lozenge$ of $J^\lozenge \to \Md$. 
\end{theorem}
Note that $\pi^\lozenge$ is almost never proper. %, and the section $\sigma$ will almost never extend over the whole of $\MD$. 
\begin{proof}
Toric varieties in this sense are always normal, see \cite{Fulton1993Introduction-to}. The properties of $\pi^\loz$ are all local on the target, so it is enough to check them for the $\Md_\Gamma \to \Mbar_\Gamma$. Separatedness and logarithmic-\'etaleness are automatic for toric varieties (in the sense of Fulton), finite presentation follows from the finiteness of the fans, see \ref{cor:form_fan}. The maps are clearly isomorphisms over the locus $\ca M$ of smooth curves, hence are birational. 
%\item Indeed, we constructed them by descending schemes. 
%\item Since properness is local on the target, this follows from the completeness of the fan $\bar F_\Gamma$ as shown in \ref{lem:bar_f_nice}. 
%\item Again it is enough to check this for the $\Md_\Gamma$ and $\MD_\Gamma$, which are toric varieties in the sense of Fulton, so automatically normal. 
%\item By uniqueness it is enough to check this locally. The proof of the local case is postponed to  \ref{lem:w_aligned_implies_extending}. 
%\end{enumerate}
\end{proof}

\section{Universal property of $\Md$}\label{sec:universal_property_md}
%\todo{Check signs in this section - go through a simple example with a text curve. }
\begin{definition}\label{def:desingularisable}
We say a stack $T$ is \emph{locally desingularisable} if \'etale-locally it admits a proper surjective finitely presented map $T' \to T$ with $T'$ regular and $T' \to T$ inducing an isomorphism between some dense open substacks of $T'$ and $T$. 
\end{definition}
For example, this is true in characteristic zero, and for arithmetic surfaces and threefolds, and for log regular schemes. We can now give a more precise variant of the notion of a \emph{$\sigma$-extending} morphism from the introduction. Recall that $\sigma\colon \ca M \to \ca J$ is the morphism given by $\omega^{\otimes k}(-\sum_i a_i x_i)$. 

Given any map $t\colon T \to \Mbar$, we can consider the open subset $t^{-1}\ca M = T \times_{\Mbar}\ca M \hra T$, and the semiabelian scheme $t^*\ca J = T \times_{\Mbar}\ca J$ over $T$. We denote by $t^*\sigma$ the canonical map $id_T \times \sigma\colon t^{-1}\ca M \to t^* \ca J$.

\begin{definition}\label{def:sigma_extending}
We say a map $t\colon T \to \Mbar$ is \emph{non-degenerate} if $T$ is normal and locally desingularisable, and $t^{-1}\ca M$ is dense in $T$. We say $t$ is \emph{$\sigma$-extending} if in addition the section $t^*\sigma \in (t^*\ca J)(t^{-1}{\ca M})$ admits a (necessarily unique) extension to $t^*\ca J(T)$. 
\end{definition}

For example, the open immersion $\ca M \to \Mbar$ is clearly $\sigma$-extending, and in general the identity on $\Mbar$ is not $\sigma$-extending. In this section we will show that $\Md \to \Mbar$ is $\sigma$-extending, and moreover that $\Md$ is universal with respect to this property: that it is terminal in the 2-category of $\sigma$-extending morphisms (\ref{cor:universal_property_md}). Our main technical result is the following: 
\begin{lemma}\label{lem:w_aligned_implies_extending}
Fix a combinatorial chart $\Mbar \leftarrow U \to \Mbar_\Gamma$, and let $t\colon T \to U$ be such that the composite $T \to \Mbar$ is non-degenerate. The following are equivalent: 
\begin{enumerate}
\item 
Locally on $T$ there exists a weighting $w$ on $\Gamma$ such that $T \to \Mbar_\Gamma$ factors via $\Md_w \to \Mbar_\Gamma$. 
\item 
$T \to \Mbar$ is $\sigma$-extending. 
\end{enumerate}
\end{lemma}
Before giving the proof we set up a little notation, which will also be useful in \ref{sec:properness_of_DRL}. The map $T \to \Mbar$ defines a stable curve over $T$, which we shall denote by $C$. For each edge $e \in E = E(\Gamma)$ we define $\ell(e) \in \ca O_{\Mbar_\Gamma}(\Mbar_\Gamma)$ to be the image of the standard basis vector $\delta_e \in \bb N^E$ under the log structure map. If $w\in W(\Gamma)$ is a weighting and $\gamma \sub \Gamma$ a cycle, we let
\begin{equation*}
\delta_\gamma = \prod_{e \in \gamma} \delta_e^{w_\gamma(e)} \in \bb N^E
\end{equation*}
where $w_\gamma(e) \in \bb Z$ is the value of $w$ on $e$ in the direction dictated by $\gamma$. %We define $\ell(\gamma)\in \ca O_{\Mbar_\Gamma}(\Mbar_\Gamma)$ to be the image of $\delta_\gamma$ under the log structure map. 
One easily verifies
\begin{lemma}\label{lem:cone_span_lines}
The dual cone $c_w^\vee \sub \bb Z^E$ is the span of the positive orthant in $\bb Z^E$ together with the $\delta_\gamma$ for $\gamma$ running over cycles in $\Gamma$. 
\end{lemma}

\begin{proof}[Proof of \ref{lem:w_aligned_implies_extending}]
\leavevmode

$(1) \implies (2)$: We may assume $T$ is local, since an extension is unique if it exists. Perhaps shrinking the combinatorial chart we may assume that $\Gamma$ is the graph over the closed point of $T$. %We write $C_T/T$ for the pullback of the universal curve from $\Mbar$, write $\Sigma$ for the tautological divisor, and $\sigma = \ca O_T(\Sigma)$ for the induced section over $t^{-1} { \ca M}$. 
%We need to extend $\sigma$ to a section of $t^*J$ over the whole of $T$. 

The map $T \to \Mbar_\Gamma$ corresponds to a map $t^\#\colon \Lambda[\bb N^E] \to \ca O_T(T)$, and each $t^\#\delta_e$ is a unit on $t^{-1}\ca M$. For each cycle $\gamma$ we obtain an element $t^\#\delta_\gamma \in \on{Frac}\ca O_T(T)$, and the factorisation of $t$ via $\Md_w$ says that $t^\#\delta_\gamma \in \ca O_T(T) \sub \on{Frac}\ca O_T(T)$. If we write $i(\gamma)$ for the cycle with the same edges as $\gamma$ but in the reverse direction, we see that $t^\#\delta_{i(\gamma)} = t^\#\delta_\gamma^{-1}$ in $\on{Frac}\ca O_T(T)$, and hence that actually each $t^\#\delta_\gamma$ is a unit, i.e. lies in $\ca O_T(T)^\times \sub \ca O_T(T)$. 

Since the product around each cycle in $\Gamma$ of the $t^\# \delta_e^{w_\gamma(e)}$ lies in $\ca O_T(T)^\times$, we can choose elements $r_v \in \on{Frac}\ca O_T(T)^\times$ for each vertex $v$ of $\Gamma$ such that for each directed edge $e\colon u \to v$ we have 
\begin{equation}\label{eq:r_label_relation}
\frac{r_u}{r_v} = t^\#\delta_e^{w_\gamma(e)}\cdot (\text{unit in }\ca O_T(T)). 
\end{equation}
For a vertex $v$, write $\eta_v$ for the generic point of the component of the special fibre of $C$ corresponding to $v$. Now we define a Weil divisor $Y$ on $C$ by specifying that $Y$ is trivial over $t^{-1}\ca M$, and that locally around $\eta_v$ it is cut out by $r_v$. Then \ref{eq:r_label_relation} and a small computation implies that $Y$ is actually a Cartier divisor. % (a small computation). %, or cf. \cite{Holmes2014Neron-models-an}). 

Now we claim that the line bundle $\omega^{\otimes k}_C(\Sigma + Y)$ defines an extension of $\sigma$ in $t^*J = \on{Pic}^0_{C/T}$. Clearly it coincides with $\sigma$ over $t^{-1} \ca M$, so all we need to check is that $\Sigma + Y$ has degree zero on every irreducible component of the special fibre of $C$. Fix a vertex $v$. We need to check that the degree of $\ca O_C(Y)$ on the component $C_v$ of the special fibre corresponding to $v$ is exactly the sum of the weights of the non-leg half-edges out of $v$. After adjusting $Y$ by the pullback of a divisor on $T$ we may assume that $r_v = 1$. Let $e\colon u \to v$ be an edge out of $v$. Then the completed \'etale local ring at the singular point corresponding to $e$ is isomorphic to
\begin{equation*}
\widehat{\ca O_T(T)^{et}}[[x,y]]/(xy- t^\#\delta_e), 
\end{equation*}
where we take $x$ to be the coordinate vanishing on $C_v$ and $y$ to be vanishing on $C_u$. We may assume that $r_u$ is given by $t^\#\delta_e^{w(e)}$ (since we can ignore $T$-units), so $Y$ is locally defined by $y^{w(e)}$, and the order of vanishing on $C_v$ is exactly $w(e)$ as required. 

$(2) \implies (1)$: Again, we may assume $T$ is local. We consider first the case where $T$ is regular. The argument above is almost reversible; we are assuming this extension of $\sigma$ exists, and it is necessarily given by a line bundle $\cl L$ of degree 0 on every irreducible component of every fibre. Then $\cl L (-\Sigma) = \ca O(Y)$ for some vertical Cartier divisor $Y$. For each vertex $v$ of $\Gamma$, let $r_v \in \on{Frac}\ca O_T(T)^\times$ be a local equation for $Y$ at $\eta_v$. 

By the regularity of $T$ we can apply \cite[theorem 4.1]{Holmes2014Neron-models-an} to see that for each directed edge $e\colon u \to v$, an equation of the form 
\begin{equation}
r_u/r_v = t^\#\delta_e^{a}\cdot (\text{unit in }\ca O_T(T))
\end{equation}
holds for some $a \in \bb Z$. Assigning to the directed edge $e$ the integer $a$ is easily verified to give a weighting $w$ on $\Gamma$, and for each cycle $\gamma$ we see that $\prod_{e \in \gamma} t^\# \delta_e^{w_\gamma(e)} \in \on{Frac}\ca O_T(T)^\times$ actually lies in the subgroup $\ca O_T(T)^\times \sub \on{Frac}\ca O_T(T)^\times$. Hence the map
\begin{equation*}
\begin{split}
\Lambda[c_w^\vee] & \to \on{Frac}\ca O_T(T)\\
\delta_e &  \mapsto  t^\# \delta_e
\end{split}
\end{equation*}
factors via $\ca O_T(T) \sub \on{Frac}\ca O_T(T)$, and we are done. 

It remains to reduce the general case to the case when $T$ is regular. So assume $T$ is local, normal and (locally) desingularisable, and let $T' \to T$ be a desingularisation, so $T'$ is regular and $T' \to T$ is proper, surjective, and birational. By Zariski's Main Theorem, the fibres of $T' \to T$ are connected. Write $\frak t$ for the closed point of $T$.  

Since we know the result in the regular case, we can apply this to $T'$ to find an open cover $\{V_i\}_{i \in I}$ of $T'$, and weightings $w_i$ on the $V_i$, such that each $V_i \to \Mbar_\Gamma$ factors via $\Md_{w_i} \to \Mbar_\Gamma$. Adjusting the cover, we may assume that each $V_i$ is connected and meets the fibre $T'_{\frak t}$ of $T'$ over the closed point of $T$. If $\Gamma'$ is the graph over $\frak t$, the weightings $w_i$ on $\Gamma$ need not be unique, but their restrictions to the contracted graph $\Gamma'$ are unique. To simplify the notation we will assume $\Gamma' = \Gamma$, since the value taken by the weightings on the contracted edges never plays any role. 

Now $T'_{\frak t}$ is connected and is covered by the $V_i \cap T'_{\frak t}$, and the weightings $w_i$ must agree on overlaps of the $V_i \cap T'_{\frak t}$, so we see that actually all the $w_i$ are equal. Write $w$ for this weighting; we will show that $T \to \Mbar_\Gamma$ factors via $\Md_w \to \Mbar_\Gamma$. 

%Each edge $e$ of $\Gamma$ is labelled by a principal ideal in $T$ (the label can be seen as the ideal generated by the image of the relevant element of the log structure, cf. \ref{above}). Write $\ell(e)$ for the label of an edge $e$, which we can think of as a Cartier divisor since $t^{-1}{\ca M}$ is dense in $T$ (so the label cannot be the zero ideal). 

Note that each $t^\#\delta_e$ is a regular element in $\ca O_T(T)$ since its restriction to $t^{-1}\ca M$ is invertible. We write $\on{div}(t^\#\delta_e)$ for the associated Cartier divisor on $T$. 

Fix a directed loop $\gamma$ in $\Gamma$. By a similar argument as in the regular case, to construct the map $T \to \Md_w$ it is enough to show that, for each cycle $\gamma$ we have 
\begin{equation}\label{eq:cartier_equality}
\sum_{e \in \gamma} w_\gamma(e)\on{div}(t^\#\delta_e) = 0. 
\end{equation}
But we know that \ref{eq:cartier_equality} holds on each $V_i$ after pulling back (since we have maps $V_i \to \Md_w$) so the result follows from \cite[lemma 2.23]{Holmes2014Neron-models-an}. 
\end{proof}

\begin{corollary}\label{lem:extending_implies_factors}
Let $t\colon T \to \Mbar$ be any $\sigma$-extending morphism. Then $t$ factors uniquely via $\Md \to \Mbar$. 
\end{corollary}
\begin{proof}
Uniqueness is clear since the map is determined on $\ca M$, whose pullback is dense in $T$. Existence then follows immediately from \ref{lem:w_aligned_implies_extending}, since $\Md_\Gamma$ is formed by glueing together the $\Md_w$, and the uniqueness ensures that the maps we obtain glue on overlaps. 
\end{proof}

\begin{corollary}\label{cor:universal_property_md}
$\Md \to \Mbar$ is the terminal object in the 2-category of $\sigma$-extending morphisms to $\Mbar$. 
\end{corollary}
\begin{proof}
We know by \ref{thm:properties_of_Md_MD} that $\Md$ is normal, and it is locally desingularisable by \cite{Kato1994Toric-singulari} since it is locally toric (log regular). Applying \ref{lem:w_aligned_implies_extending} we see that $\Md \to \Mbar$ is $\sigma$-extending, since $\Md_w \to \Mbar_\Gamma$ is clearly $w$-aligned. \Cref{lem:extending_implies_factors} then shows that it is terminal. 
\end{proof}

\section{Properness of $\on{DRL}_\loz$}\label{sec:properness_of_DRL}
Write $\sigma'_\loz\colon \Md \to J$ for the extension of $\sigma$ (we reserve the notation $\sigma_\loz$ for the induced map $\Md\to \Md \times_{\Mbar} J$). Write $\on{DRL}_\loz$ for the schematic pullback of the unit section of the universal jacobian along $\sigma'_\loz$, so $\on{DRL}_\loz$ is a closed substack of $\Md$. We will show that $\on{DRL}_\loz$ is proper over $\Mbar$ (recalling that the map $\Md \to \Mbar$ is in general far from proper, c.f. \ref{fig:fan_Md}). 

\subsection{In characteristic zero}\label{sec:char_0_properness}
Over a field of characteristic zero (or more generally, over a ring over which all schemes of finite type admit a resolution of singularities) this result can be proven more directly. Namely, we will show that $\Md$ coincides with the normalisation of the closure of the image of $\sigma$ in $J$, from which properness follows immediately. More precisely, it is the underlying scheme of $\Md$ which is given by this construction, since the normalisation of the closure does not come with a natural logarithmic structure. 

Write $S$ for the schematic image of the section $\sigma$ in $J$ (in other words, for the closure of the image, with suitable reduced structure), and $S'$ for its normalisation. We will show that $S'$ is a universal $\sigma$-extending morphism. Indeed, the pullback of $\ca M$ is evidently dense in $S'$, and $S'$ is automatically normal, and admits a resolution of singularities by \cite{Hironaka1964Resolution-of-s} (here we use characteristic zero). Moreover, $S'$ comes with a map to $J$ making it $\sigma$-extending. 

Suppose $T \to \ca M$ is also $\sigma$-extending, then the map $T \to J$ factors via the inclusion $S \tra J$ by definition of the schematic image, and the resulting map $T \to S$ lifts to $S'$ since $T$ is assumed normal. We thus see that $S'$ is the universal $\sigma$-extending morphism, and so is canonically isomorphic to $\Md$. 

To prove properness of $\on{DRL}_\loz$ over $\Mbar$ we simply observe that $S' \to S$ is proper, so $\on{DRL}_\loz$ is proper over the intersection of $S$ with the unit section in $J$, and that intersection is automatically proper over $\Mbar$ since it is closed in the image of the unit section. 

In positive characteristic, we do not know that $S'$ is $\sigma$-extending (\ref{def:sigma_extending}), since we do not know if it admits a resolution of singularities. In the remainder of this section we will give a more hands-on proof of properness in this case. Some of the details (in particular the description of the universal line bundle) may also be of independent interest.

\subsection{The universal line bundle}\label{sec:universal_bundle}
Before giving the proof in the general case, we describe a certain line bundle on the universal curve over $\Md$, which will play a crucial role in the proof --- this may be of some independent interest. The proof itself is then mainly a matter of keeping careful track of isomorphisms and valuations. To simplify the notation, we write $\omega$ for the relative dualising sheaf of the universal curve over $\Mbar$. %, and we will abuse notation by writing $\sigma$ for the extension $\sigma'_\loz$ of the section $\sigma$ over $\Md$. 

Let $p \colon \on{Spec} k \to \Md$ be a geometric point, and write $C_p$ for the stable curve over $k$. The map $\sigma'_\loz\colon \Md \to J$ determines an isomorphism class of line bundles on $C_p$ (necessarily with degree zero on every component). In this section we will give representatives of this isomorphism class. % $[\cl F(-\Sigma)]$ (where $\Sigma$ is as always the divisor $\sum_i a_i x_i$). 

\begin{remark}\label{rem:rationality}
In fact, we do not need to assume that the field $k$ is algebraically closed; it is enough to assume that all irreducible components of $C_p$ are geometrically irreducible, and that all preimages of all nodes in the normalisation of $C_p$ are $k$-rational points. 
\end{remark}

Choose a combinatorial chart $\Mbar_\Gamma \leftarrow U \to \Mbar$ containing $p$, and such that $\Gamma$ is the graph of $C_p$. Let $w$ be the weighting on $\Gamma$ such that $p$ lies in $\Md_w$. If $v$ is a vertex of $\Gamma$, write $C_v$ for the corresponding irreducible component of $C_p$, and define a line bundle $\cl F_v$ on $C_v$ by the formula
\begin{equation}\label{eq:simple_f_formula}
\cl F_v = \omega^{\otimes k}|_{C_v}\left(\sum_{e}w(e)\cdot [e_v]\right). 
\end{equation}
Here the sum runs over directed edges $e$ out of $v$, and $e_v$ is the point on $C_v$ corresponding to the node $e$ on $C_p$. If $e$ is a self-loop then the point $e_v$ is not a Cartier divisor on $C_v$, but (by the definition of a weighting) it appears with coefficient $0$ in the above formula, so we do not worry about this. Writing $\tilde C_v$ for the normalisation of $C_v$, we have $\omega|_{\tilde C_v} = \Omega_{\tilde C_v}(\sum_e [e_v])$ (where the sum runs over all directed edges with an end at $v$), and the restriction of the latter to any node $e_v$ is canonically trivialised by the residue map. 

Now we need to glue the $\cl F_v$ together along the non-self-loop edges $e$ to give a line bundle on $C_p$. If $e\colon u \to v$ is a non-self-loop edge then we see
\begin{equation}\label{eq:triv_hom}
\begin{split}
\on{Hom}_k(\cl F_v|_{e_v}, \cl F_u|_{e_u}) 
& = \Big(\omega^{\otimes k}|_{C_v}(w(e)[e_v])\Big)\Big|_{e_v}^{\otimes -1} \otimes  \Big(\omega^{\otimes k}|_{C_u}(-w(e)[e_u])\Big)\Big|_{e_u} \\
& = \ca O_{C_v}(w(e)[e_v])|_{e_v}^{\otimes -1} \otimes  \ca O_{C_u}(-w(e)[e_u])|_{e_u} \\
& = \Big( \ca O_{C_v}([e_v])|_{e_v}\otimes \ca O_{C_u}([e_u])|_{e_u}\Big)^{\otimes-w(e)}. 
\end{split}
\end{equation}
By the deformation theory of stable curves the vector space $\ca O_{C_v}([e_v])|_{e_v}\otimes \ca O_{C_u}([e_u])|_{e_u}$ is naturally a sub-space of the tangent space to $p$ in $\Mbar$. The choice of combinatorial chart then yields a canonical generator of this summand of the tangent space, giving us a canonical isomorphism 
\begin{equation}\label{eq:triv_ten}
\ca O_{C_v}([e_v])|_{e_v}\otimes \ca O_{C_u}([e_u])|_{e_u} \iso k. 
\end{equation}
We can use this to explicitly describe how to glue the $\cl F_v$ together to a line bundle on $C_p$. The map $p\colon \on{Spec} k \to \Md_w$ corresponds to a map $p^\#\colon \Lambda[c_w^\vee] \to k$, and for each cycle $\gamma$ we see $p^\#\delta_\gamma \in k^\times$. Choose a function 
\begin{equation}\label{eq:lambda}
\lambda\colon \stackrel{\to}{E} \to k^\times
\end{equation}
such that $\lambda(i(e)) = \lambda(e)^{-1}$ and such that for every cycle $\gamma$ we have 
\begin{equation}\label{eq:prod}
\prod_{e \in \gamma} \lambda(e)^{w(e)} = p^\#\delta_\gamma. 
\end{equation}
For an edge $e\colon u \to v$, the $-w(e)$th power of the element $\lambda(e) \in k^\times$ gives (via \ref{eq:triv_ten} and \ref{eq:triv_hom}) an isomorphism $\cl F_v|_{e_v} \iso \cl F_u|_{e_u}$. We use these isomorphisms to glue the $\cl F_v$ to a line bundle on the whole of $C_p$ (cf. \cite{Ferrand2003Conducteur-desc}), which we denote $\cl F_\lambda$. Clearly $\cl F_\lambda$ depends on the choice of $\lambda$, but a different choice of $\lambda$ will yield an isomorphic $\cl F_\lambda$. 

\begin{proposition}\label{prop:univ_correction_bundle}
The isomorphism class $[\cl F_\lambda(-\Sigma)]$ of line bundles on $C_p$ corresponds to the image of the section $\sigma'_\loz$ in $\on{Pic}^0_{C_p/k}$. 
\end{proposition}
\begin{proof}
This follows by a rather messy unravelling of the constructions in \ref{sec:construction}. 
\end{proof}

\subsection{The proof of properness in general}\label{sec:proper_general_case}

\begin{proposition}\label{thm:properness_lozenge} \label{lem:vcp}
The map $\on{DRL}_\loz \to \Mbar$ is proper. 
\end{proposition}

\begin{proof}
\leavevmode
\noindent \textbf {Step 1:} Setup. 

The map is clearly separated and of finite presentation since the same holds for $\Md \to \Mbar$. We need to show that the dashed arrow in the following diagram can be filled in: 
\begin{equation}\label{tikz:proper}
\begin{tikzcd}
\eta \arrow[r] \arrow[d] & \on{DRL}_\loz \arrow[d] \\%\arrow[r, tail]&  U_Z^\lozenge  = (\bb G_m^r)_Z\arrow[dl]\\
T \arrow[r, "t"]\arrow[ur, dashed]  & \Mbar \\
\end{tikzcd}
\end{equation}
where $T$ is a strictly hensellian trait with generic point $\eta$ and closed point $p$. Choose a combinatorial chart $\Mbar_\Gamma \leftarrow U \to \Mbar$ containing $p$, and such that $\Gamma = \Gamma_p$ is the graph of $C_p$. 

We write $\Gamma_\eta$ for the graph over $\eta$, with edge set $E_\eta$ and vertex set $V_\eta$, and similarly over $p$, so we have a contraction map $\Gamma_p \to \Gamma_\eta$, and $E_\eta \sub E_p$. Given $v \in V_\eta$ we define $C_v$ to be the corresponding irreducible component of $C_\eta$. 

Let $w_\eta$ be a weighting on $\Gamma_\eta$ such that $\eta$ lands in $\Md_{w_\eta}$. 

\noindent \textbf {Step 2:} Extending the weighting to $\Gamma_p$. 

Because the trait $T$ is strictly hensellian, we know that all irreducible components of $C_\eta$ are geometrically irreducible, and all nodes and their tangent directions are rational over $\eta$. Because of this, by \ref{rem:rationality}, we can apply \ref{prop:univ_correction_bundle} over $\eta$. Accordingly, we choose
%Applying \ref{prop:univ_correction_bundle} over $\eta$ we choose: 
\begin{itemize}
\item
a line bundle $\cl L$ on $C_\eta$; 
\item 
for each $v \in V_\eta$ an isomorphism 
\begin{equation}\label{iso_3}
\cl L|_{C_v} \iso \omega^{\otimes k}|_{C_v}(\sum_v w(e)[e_v]) = \cl F_v. 
\end{equation}
\item 
an isomorphism 
\begin{equation}\label{iso_2}
\cl L(-\Sigma) \iso \ca O_{C_\eta};
\end{equation}
\end{itemize}
(the last is possible exactly because $\eta$ lands in $\on{DRL}_\loz$). 
Putting together \ref{iso_2} and \ref{iso_3} we obtain for each $e\colon u \to v$ an isomorphism 
\begin{equation}\label{eq:iso}
\cl F_v|_{e_v} \iso \cl F_u|_{e_u}. 
\end{equation}
Perhaps after replacing $\eta$ by a finite extension, we can choose elements $\lambda(e) \in \ca O_T(\eta)$ as in \ref{eq:lambda}, satisfying \ref{eq:prod} and inducing via \ref{eq:triv_hom} and \ref{eq:triv_ten} the isomorphisms \ref{eq:iso}. 

We write $\bar C_v$ for the closure of $C_v$ in the curve $C = C_T$ (so the $\bar C_v$ are the irreducible components of $C$), and we write $\bar {\cl F}_v$ for the line bundle on $\bar C_v$ given by 
\begin{equation}\label{eq:formula_bar_f}
 \bar{\cl F}_v  = \omega^{\otimes k}|_{\bar C_v}\left(\sum_s w(e)[e_v]\right), 
\end{equation}
where now we view $e_v$ as an element of $\bar C_v(T)$ (cf. \ref{eq:simple_f_formula}). Putting together \ref{iso_2} and \ref{iso_3} again we obtain trivialisations $\cl F_v(-\Sigma)\otimes \omega^{\otimes -k} \iso \ca O_{C_v}$, which we can think of as being trivialisations of $\bar{\cl F}_v(-\Sigma)\otimes \omega^{\otimes -k}$ over $\eta$, which yield Cartier divisors $Y_v$ on $\bar C_v$ supported on the special fibre. 

If $v' \in V_p$ is a vertex mapping to $v \in V_\eta$, we define $\ca Y(v')$ to be the multiplicity along the generic point of $C_{v'}$ of the divisor $Y_v$; this gives a function $\ca Y\colon V_p \to \bb Z$. Given $e\colon u \to v \in E_p \setminus E_\eta$ we define 
\begin{equation}
w(e) = \frac{\ca Y(v) - \ca Y(u)}{\on{thickness}(e)} \in \bb Z. 
\end{equation}
One now checks easily that this $w$ extends the weighting $w_\eta$ to a weighting $w$ onto whole of $\Gamma_p$. 

\noindent \textbf {Step 3:} Computing $\on{ord}_T \lambda(e)$. 

Let $e\colon u \to v  \in E_\eta$, so we have $\lambda(e) \in \ca O_T(\eta)$ giving the glueing $\cl F_v|_{e_v} \iso \cl F_u|_{e_u}$. Over $T$ we have isomorphisms
\begin{equation*}
\bar{\cl F}_v(Y_v)|_{e_v} = \bar{\cl F}_v(-\Sigma + Y_v)|_{e_v} \stackrel{a}{\to} \omega^{\otimes k}|_{e_v} = \omega^{\otimes k}|_{e_u}\stackrel{b}{\to} \bar{\cl F}_u(-\Sigma + Y_u)|_{e_u}  = \bar{\cl F}_u(Y_u)|_{e_u},  
\end{equation*}
where the isomorphisms $a$ and $b$ come from the definition of $Y_v$. 
Now $e \in E_\eta$ lifts to a unique edge $e' \in E_p$, and we write $v' \in V_p$ for the endpoint of $e'$ which maps to $v$, and similarly define $u'$. A small computation then shows that 
\begin{equation}
\on{ord}_T\lambda(e) = \ca Y(v') - \ca Y(u'). 
\end{equation}

\noindent \textbf {Step 4:} Constructing a lift $T \to \Md_w$. 

We want to construct a map $\Lambda[c_w^\vee] \to \ca O_T(T)$ (recall the description of $c_w^\vee$ from \ref{sec:universal_property_md}). The map $\Lambda[\bb N^{E_p}] \to \ca O_T(T)$ vanishes exactly on $\delta_e$ for $e \in E_\eta \sub E_p$, so we have a map $\Lambda[\bb Z^{E_p \setminus E_\eta}] \to \ca O_T(\eta)$, which we extend to a map $\phi\colon \Lambda[\bb Z^{E_p}] \to \ca O_T(\eta)$ by sending $\delta_e$ to $\lambda(e)$ for each $e \in E_\eta$. 

Let $\gamma$ be a cycle in $\Gamma_p$. Then 
\begin{equation*}
\sum_{(e \colon u \to v) \in \gamma} \on{ord}_T \phi(\delta_e)^{w_\gamma(e)} = \sum_{(e\colon u \to v) \in \gamma} \ca Y(v) - \ca Y(u) = 0, 
\end{equation*}
so $\phi\colon \Lambda[\bb Z^{E_p}] \to \ca O_T(\eta)$ restricts to a map $\Lambda[c_w^\vee] \to \ca O_T(T)$ as required. 

\noindent \textbf {Step 5:} Verifying the valuative criterion. 

We have constructed a map $T \to \Md$ over $t$ whose restriction to $\eta$ is as in \ref{tikz:proper}. Since the inclusion $\on{DRL}_\loz \to \Md$ is proper it follows that $T \to \Md$ factors via $\on{DRL}_\loz$ as required. 
\end{proof}

\section{Proof of \ref{main_theorem:proper,main_theorem:limit} }%(properness, and existence of limit)}

We now have all the tools to easily prove \ref{main_theorem:proper,main_theorem:limit}. We begin by giving slightly more precise statements. 

An \emph{admissible modification of $\Mbar$} is a morphism $x\colon X \to \Mbar$ satisfying: 
\begin{itemize}
\item $x$ is proper and surjective; 
\item $x$ is birational (i.e. there exists a dense open $U \sub \Mbar$ such that $x^{-1}U$ is dense in $X$ and $x^{-1}U \to U$ is an isomorphism); 
\item  $X$ is normal and locally desingularisable (\ref{def:desingularisable}). 
\end{itemize}
 Admissible modifications together with maps over $\Mbar$ form a directed system. We check that $\Md$ can be compactified to an admissible modification (perhaps after modification of $\Md$ itself):

\begin{lemma}\label{lem:existence_of_cpct}
There exist a proper birational map $\widetilde \Md \to \Md$, an admissible modification $\MD \to \Mbar$, and an open immersion $\widetilde \Md \to \MD$  over $\Mbar$. In characteristic zero we may take $\widetilde \Md = \Md$. 
\end{lemma}
Note that this $\widetilde \Md$ is then the largest open of $\MD$ which admits a map to $\Md$ over $\Mbar$. 
\begin{proof}
We begin by giving a simple argument in characteristic zero. First construct some compactification $\MD$ of $\Md \to \Mbar$ following \cite[\S 6]{Rydh2011Compactificatio}, then apply Hironaka \cite{Hironaka1964Resolution-of-s} to see that $\MD$ is desingularisable hence $\MD \to \Mbar$ is an admissible modification. 

In arbitrary characteristic the proof is slightly more involved. It is enough to treat the connected components of $\Mbar$ separately, so we fix one such component; abusing notation, we will still write it as $\Mbar$, and similarly write $\Md \to \Mbar$. 

Choose a finite cover $\frak U = \{U_i\to \Mbar\}$ by combinatorial charts (each $U_i$ having graph $\Gamma_i$), then choose a finite cover $\frak V = \{V_j \to \frak U \times_{\Mbar} \frak U\}$ of $\frak U \times_{\Mbar} \frak U$ by combinatorial charts, each $V_j$ having graph $\Gamma'_j$. For each $U_i$ we have a fan $F_i$ in $\bb Q^{E(\Gamma_i)}$ from \ref{def:fans}, and similarly for each $V_j$ a fan $F_j'$ in $\bb Q^{E(\Gamma'_j)}$. 

Fix a $V_j$. Recalling that combinatorial charts are by definition connected, composing the map $V_j \to \frak U \times_{\Mbar} \frak U$ with one of the projections yields a map $V_j \to U_i$ for some $i$. This yields maps $\Gamma_i \to \Gamma'_j$, $E(\Gamma_j') \to E(\Gamma_i)$, and $\bb Q^{E(\Gamma'_j)} \to \bb Q^{E(\Gamma_i)}$, and $F_i$ necessarily pulls back to $F'_j$ (we say the $F_i$ are `compatible on overlaps'). 

For each $i$ we can choose a finite refinement $\bar F_i$ of $F_i$ which is complete in the sense that it fills the positive orthant. Hence if $\MD_i/U_i$ is the toric variety associated to $\bar F_i$ then the map $\MD_i \to U_i$ is proper. If we write $\tilde F_i$ for the fan obtained by restricting $\bar F_i$ to the support of $F_i$ then the associated toric variety $\widetilde \Md_i$ over $U_i$ is a blowup of $\Md_{U_i}$ and has a natural open immersion to $\MD_i$. Since toric varieties are desingularisable we have proven the lemma `locally on $\Mbar$'. 

Since the transition maps $\bb Q^{E(\Gamma'_j)} \to \bb Q^{E(\Gamma_i)}$ are just inclusions of coordinate subspaces, it is not hard to check that the $\bar F_i$ can be chosen to be compatible on overlaps, whence they will glue to give a global construction. 
%
%
%If we can choose the $\bar F_i$ to be compatible on overlaps then we can descend them to construct $\MD$ as a log smooth algebraic space over $\Mbar$ (glueing maps are determined by their restrictions to $\Md$, hence are unique if they exist). We begin by choosing any collection of $\bar F_i$ as in step 1, then refine them until they become compatible. 
%We define a \emph{coordinate subspace} of $\bb Q_{\ge 0}^n$ to be a cone cut out by setting some number of the coordinate to zero. Note that $\frak U$ and $\frak V$ are finite and the glueing maps $\bb Q_{\ge 0}^{E(\Gamma_j')} \to \bb Q_{\ge 0}^{E(\Gamma_i)}$ consist of inclusions of coordinate subspaces. 
%
%We will construct a sequence of fans $G_n$ filling the positive orthant of $\bb Q^n$. Construct $G_1$ by superimposing all 1-dimensional coordinate subspaces of all the $\bar F_i$. Given $G_{n-1}$, define $G_n$ by superimposing the following fans: 
%\begin{itemize}
%\item
%All $n$-dimensional coordinate subspaces of all the $\bar F_i$; 
%\item The fan $\hat G_{n-1}$ constructed as follows: for each inclusion $\phi\colon \bb Q_{\ge 0}^{n-1} \hra \bb Q_{\ge 0}^n$ as coordinate subspace, and for each cone $c \in G_{n-1}$, let $c_\phi$ denote the cone spanned by $\phi(c)$ and the vector $(1, \dots, 1)$ in $\bb Q_{\ge 0}^n$. Then $\hat G_{n-1}$ consists of all the $c_\phi$. 
%\end{itemize}
%If $\bar F_i$ lives in $\bb Q_{\ge 0}^n$ then we choose the refinement $G_n$ of $\bar F_i$. It is then easy to see that the resulting collection of refinements of the $\bar F_i$ are compatible on overlaps. 
\end{proof}

\begin{definition}\label{def:DRL_DRC}
Given an admissible modification $x\colon X\to \Mbar$, the map $\sigma$ induces a map $\sigma'\colon X \times_{\Mbar} \ca M \to J$. We write $\mathring X$ for the largest open of $X$ over which the closure of the graph of $\sigma'$ inside $X \times_{\Mbar}J$ is flat. We write $\sigma_X'\colon \mathring X \to J$ for the induced map (informally, we call $\mathring X$ the `largest open of $X$ over which $\sigma$ extends'). 
%we write $\mathring X$ for the largest open over which the section $\sigma\colon \ca M \to J$ extends, and we write $\sigma_X\colon X \to J_X$ for the extension.
%More formally, we take the schematic image of $\sigma$ in $J \times_{\Mbar} X$, and define $\mathring X$ to be the locus of points in $X$ over which this schematic image is flat.
 Clearly $\mathring X$ contains $x^{-1}\ca M$. We write $e$ for the unit section in the jacobian, and $\on{DRL}_X \coloneqq (\sigma_X')^*e$ for the pullback as a closed subscheme of $\mathring X$.

 The induced map $\sigma_X\colon \mathring X \to J \times_{\Mbar}\mathring X$ is then a regular closed immersion, since $J \times_{\Mbar}\mathring X$ is smooth over $\mathring X$. We denote by $e_X$ the unit section $\mathring X \to J \times_{\Mbar}\mathring X$. Then we can pull back the fundamental class of $e_X$ to a cycle class $\on{DRC}_X\coloneqq \sigma_X^![e_X]$ on 
 \begin{equation*}
 \mathring X \times_{\sigma_X, J \times_{\Mbar} \mathring X, e} \mathring X  = \on{DRL}_X. 
 \end{equation*}
\end{definition}

\begin{remark}\label{rem:commutativity}
In the above construction, we could alternatively have defined $\on{DRC}_X$ by pulling back the class of $\sigma_X$ along $e_X$. The resulting cycle would be the same. To see this, note first that, since $\sigma_X$ and $e_X$ are regular closed immersions, the fundamental class of $\mathring X$ is exactly the refined gysin pullback (along either $\sigma_X$ or $e_X$) of the fundamental class of $J \times_{\Mbar}\mathring X$. We then apply the commutativity of the intersection pairing, see \cite[theorem 3.13]{Vistoli1989Intersection-th}, c.f. \cite[theorem 6.4]{Fulton1984Intersection-th}. 
\end{remark}
 % $\on{DRC}_X \coloneqq \sigma_X^*[e]$ for the pullback\todo{Its not clear to me what this means, if things are rather singular. Instead, take the construction we are already using tin the proof of \ref{main_theorem:limit_full}). Is it the `correct' thing?  } as an algebraic cycle class on $\on{DRL}_X$. 
 
\begin{remark}
In general $\mathring X$ is strictly smaller than $X$. For example, if $X = \bar {\ca M}_{g,n}$ itself then $\mathring {\bar {\ca M}}_{g,n}$ contains the compact type (even tree-like) locus, but need not be the whole of $\bar {\ca M}_{g,n}$. If $g=1$, $n=2$, and $a_1 = - a_2 = 2$ then $\mathring {\bar {\ca M}}_{g,n}$ is the complement of the boundary point corresponding to a pair of $\bb P^1$s glued at $0$ and $\infty$. If we take $X$ to be the blowup of $\Mbar_{g,n}$ at that point then $\mathring X$ will be the complement of the two singular points in the boundary that are contained within the exceptional curve of the blowup. 
\end{remark}

 Given any $\MD$ as in \ref{lem:existence_of_cpct}, note that $\mathring \MD = \widetilde \Md$. The inclusion $\mathring \MD \supset \widetilde \Md$ holds because $\Md$ and hence $\widetilde \Md$ are $\sigma$-extending. For the other implication, let $x \in \mathring \MD$, and let $T$ be a trait in $\mathring \MD$ through $x$ with generic point lying over $\ca M$. Since $T$ is $\sigma$-extending it lifts to $\Md$, and because $\widetilde \Md \to \Md$ is proper it lifts further to $\widetilde \Md$, so $x \in \widetilde \Md$. 

\begin{theorem}\label{main_theorem:proper_full}
Choose any $\MD$ as in \ref{lem:existence_of_cpct}. Let $X \to \Mbar$ be an admissible modification which factors via $\MD \to \Mbar$. Then $\on{DRL}_X \to \Mbar$ is proper. 
\end{theorem}
%{Do we need that $\tilde \Md = \mathring \MD$?}
\begin{proof}
%The map $x$ restricts to a proper map $\on{DRL}_X \to \on{DRL}_\loz${why?}, so the result follows from \ref{thm:properness_lozenge}. 
Write $f\colon X \to \MD$ for the factorisation. Since $\Md$ is $\sigma$-extending we see that $f^{-1}\widetilde \Md \sub \mathring X$. In fact they are equal; let $x \in \mathring X$, and choose a trait $T$ in $\mathring X$ through $x$ with generic point lying over $\ca M$. Then $T$ is $\sigma$-extending hence lifts to $\Md$, and it lifts further to $\widetilde \Md$ since $\widetilde \Md \to \Md$ is proper. 

Write $\mathring f\colon \mathring X \to \Md$ for the restricted map (which is proper since $f$ is proper). Then $\on{DRL}_X = \mathring f^{-1}\on{DRL}_\loz$, so we have a proper map $\mathring f \colon \on{DRL}_X \to \on{DRL}_\loz$. By \ref{thm:properness_lozenge} we know $\on{DRL}_\loz \to \Mbar$ is proper, so we are done. 
%Hence the composite 
%\begin{equation*}
%\on{DRL}_X \to \on{DRL}_\loz \to \Mbar
%\end{equation*}
%is proper as required. 
\end{proof}
Observing that the admissible modifications which factor via $\MD$ form a cofinal system among all admissible modifications, we have established \ref{main_theorem:proper}. 

Whenever $\on{DRL}_X \to \Mbar$ is proper, we can form the pushforward of $\on{DRC}_X$ to $\Mbar$; write $\pi_*\on{DRC}_X$ for this cycle on $\Mbar$. We defined $\on{DRL}_\loz$ at the beginning of \ref{sec:properness_of_DRL} as the schematic pullback of the unit section $e$ of $J$ along the section $\sigma'_\loz\colon \Md \to J$, just as in \ref{def:DRL_DRC}. 

\begin{definition}
Write $\sigma_\loz\colon \Md \to \Md \times_{\Mbar} J$ for the section induced by $\sigma'_\loz$, a regular closed immersion, and write $[e_\loz]$ for the fundamental class of $\Md$, viewed as a cycle on $\Md \times_{\Mbar} J$ via the map $e_\loz$. Then define $\on{DRC}_\loz \coloneqq \sigma_\loz^![e_\loz]$, as a class on $\on{DRL}_\loz$, the refined gysin pullback in the pullback diagram 
\begin{center}
\begin{tikzcd}
\on{DRL}_\loz \arrow[r, "\sigma'"] \arrow[d, "e'"] & \Md \arrow[d, "e_\loz"]\\
\Md \arrow[r, "\sigma_\loz"] &\Md \times_{\Mbar} J,  \\
\end{tikzcd} 
\end{center}
\end{definition}
This is exactly the cycle as we would obtain applying \ref{def:DRL_DRC} with $X = \MD$. Since $\on{DRC}_\loz$ is a cycle class on $\on{DRL}_\loz$, we can define $\overline{\on{DRC}}$ as the pushforward of this cycle to $\Mbar$. 

%the cycle-theoretic pullback\todo{How is this defined?? Again, use instead the construction from the proof of \ref{main_theorem:limit_full}. }
%If you want to use refined Gysin, need the map from $\on{DRL}$ to $J$ to be a regular embedding, which it isn't. So maybe need to resolve $\ca M^\loz$ (globally)? What happens in the paper with J? We have a smooth scheme over $\Md$, and two sections, and we pull back the class of one along the class of the other. This does seem OK. But need to compare to what we say here, and maybe correct the latter. Also care needed then in the later section comparing to the VFC. } 

\begin{theorem}\label{main_theorem:limit_full}
Choose any $\MD$ as in \ref{lem:existence_of_cpct}. If an admissible modification $x \colon X\to \Mbar$ factors via $\MD$, then $\pi_*\on{DRC}_X = \overline{\on{DRC}}$. 
\end{theorem}
This would be a formality if we could pull back the cycle $\on{DRC}_\loz$ to $X$, but since $\Md$ and $X$ may be singular we must take some care with pulling back cycles. 
\begin{proof}
%Write $x\colon X \to \MD$ for the factorisation. We first consider what would happen if $X$ and $\MD$ were smooth and $x$ were a sequence of modifications at smooth centres; then we would have a pullback $x^*\colon \on{CH} \MD \to \on{CH} X$, and we know $x_*x^*  = \on{id}$. Moreover, we easily check that $x^*\on{DRL}_\loz = \on{DRL}_X$ (because pullbacks of cycles are functorial), hence $x_*\on{DRL}_X = \on{DRL}_\loz$, and we would be done. However, in reality $\MD$ is not smooth, and we cannot assume $X$ to be smooth, so a little more care is needed. 
For simplicity we give the proof only in the case where $\widetilde \Md \to \Md$ is an isomorphism, i.e. we have an open immersion $\Md \to \MD$ (this can always be arranged in characteristic zero). To extend to the general case one uses a very similar argument to show that $\pi_*\on{DRC}_\loz = \pi_*\on{DRC}_{\MD}$. 

As in the proof of \ref{main_theorem:proper_full}, the factorisation $f\colon X \to \MD$ restricts to a proper map $\mathring f \colon \mathring X \to \Md$. Write $J_{\Md} \coloneqq J \times_{\ca M} \Md$, and $J_{\mathring X} \coloneqq J \times_{\ca M} \mathring X$. Then we can view $\sigma_{\Md}$ as a section of the projection $J_{\Md} \to \Md$, and similarly for $\mathring X$. And we write $e_{\Md}$ for the unit section of $J_{\Md}$, and similarly for $\mathring X$. Writing $f_J$ for the induced map $J_X \to J_{\Md}$, we obtain a commutative diagram
\begin{center}
\begin{tikzcd}[sep = huge]
J_{\mathring X} \arrow[r, "f_J"] \arrow[d] & J_{\Md} \arrow[d]\\
X \arrow[r, "f"] \arrow[u, bend left, "\sigma_X"] \arrow[u, bend right, swap, "e_X"] & \Md \arrow[u, bend left, "\sigma_{\Md}"] \arrow[u, bend right, swap, "e_{\Md}"]  \\
\end{tikzcd} 
\end{center}
Note that the upward-pointing arrows are closed immersions, so we can see them as algebraic cycles. Since $f$ is proper and birational, we see immediately that ${f_J}_*[e_X] = [e_{\Md}]$. By the commutativity of proper pushforward and the refined Gysin homomorphism (\cite[theorem 3.12]{Vistoli1989Intersection-th}, c.f. \cite[theorem 6.2(a)]{Fulton1984Intersection-th}), we see that 
\begin{equation*}
\sigma_{\Md}^!{f_J}_*[e_X] = f_* \sigma_X^![e_{X}]
\end{equation*}
hence %\todo{OK, so maybe this is the nicer approach; work with the fibred product, then we're pulling back along regular embeddings. Then maybe we should rephrase slightly what is written above (only defs, no lemmas, so trivial to do). Then think carefully about the comparison to the virtual class, but I'm basically sure that will be fine. }
\begin{equation*}
\on{DRC}_\loz = \sigma_{\Md}^![e_{\Md}] =\sigma_{\Md}^!{f_J}_*[e_X] = f_* \sigma_X^![e_{X}] = f_*\on{DRC}_X. \qedhere
\end{equation*}
\end{proof}
This immediately implies \ref{main_theorem:limit}, and shows moreover that the limit is given by the pushforward of $\on{DRL}_\loz$.

\section{Proof of \ref{main_theorem:comparison}}% (comparison to the cycle of Graber and Vakil)}

In this section, we will use results of Cavalieri, Marcus, and Wise (\cite{Cavalieri2011Polynomial-fami}, \cite{Marcus2013Stable-maps-to-}) to check that our double ramification cycle $\overline{\on{DRC}}$ coincides with the cycle constructed in \cite{Graber2005Relative-virtua} and computed in \cite{Janda2016Double-ramifica}. We will temporarily denote the latter cycle by $\on{DRC}_{LGV}$. Their construction only applies in the `non-twisted' case $k=0$, so for the remainder of this section we restrict to that case without further comment.

%Most of the work has already been done, either in this paper or by others; we combine in a fairly formal manner results from \cite{Marcus2013Stable-maps-to-}, Costello's theorem \cite{Costello2006Higher-genus-Gr}, and the universal property of $\Md$ proven in \ref{sec:universal_property_md}. 

We begin by briefly recalling the setup and notation from \cite{Marcus2013Stable-maps-to-}. We denote by $J$ the universal (semiabelian) jacobian over $\Mbar$. Over $\Mbar$ they construct a stack $\bar{M}(\cl P)$ of stable maps to (expansions of) $\cl P \coloneqq [\bb P^1 / \bb G_m]$. %Of course the details of the expansions and stability condition \cite[\S 2 and \S 3.1]{Marcus2013Stable-maps-to-} are of crucial importance for their results, but the details are not needed for our arguments. %, as we summarise the properties we need below. 
We will see in the next proposition that $\bar{M}(\cl P) \to \Mbar$ is birational, and their constructions yield a map $\bar{M}(\cl P) \to J$, extending the map $\sigma$ on $\ca M$ (if $\bar{M}(\cl P)$ were normal we would call it $\sigma$-extending). Writing $e \colon \Mbar \tra J$ for the closed immersion from the unit section (\cite{Marcus2013Stable-maps-to-} denote this by $Z$), we define $\bar{M}(\cl P / B\bb G_m)$ to be the fibre product of $e$ and $\bar{M}(\cl P)$ over $J$. %In \cite{Marcus2013Stable-maps-to-} this $\bar{M}(\cl P / B\bb G_m)$ is realised as a suitable stack of stable expanded maps to a `rubber $\cl P$', but we do not need this description. 

Since $e \colon \Mbar \tra J$ is a regular closed immersion (say with ideal sheaf $\ca I$), it comes with a natural perfect relative obstruction theory, namely the cotangent complex $L_{e/J} = \iota^*\ca I /\ca I^2 [1]$ itself. The associated relative virtual class is just the fundamental class $[e]$ as a cycle on $e$ (c.f. \cite[example 7.6]{Behrend1997The-intrinsic-n}). 

In \cite{Cavalieri2011Polynomial-fami} they construct directly a perfect relative obstruction theory for $\bar{M}(\cl P/B\bb G_m) \to \bar{M}(\cl P)$, and show (see \cite[remark after prop 3.5]{Cavalieri2011Polynomial-fami}) that it coincides with the pullback of $L_{e/J}$ in the fibre product diagram 
\begin{center}
\begin{tikzcd}
\bar{M}(\cl P/B\bb G_m) \arrow[r] \arrow[d, tail] & \Mbar \arrow[d, tail, "e"]\\
\bar{M}(\cl P) \arrow[r] &J. \\
\end{tikzcd} 
\end{center}
The remaining key properties we need are summarised below. 
\begin{proposition}\leavevmode
\begin{enumerate}
\item $\bar{M}(\cl P) \to \Mbar$ is birational. 
\item $\bar{M}(\cl P/B\bb G_m)$ is proper over $\Mbar$; 
\item The pushforward along $\bar{M}(\cl P/B\bb G_m) \to \Mbar$ of the class associated to the pullback of $L_{e/J}$ to $\bar{M}(\cl P/B\bb G_m) \to \bar{M}(\cl P)$ is the double ramification cycle $\on{DRC}_{LGV}$ of \cite{Graber2005Relative-virtua}. 
%\item $\bar{M}(\cl P) \to \Mbar$ is log etale. 
\end{enumerate}
\end{proposition}
\begin{proof}\leavevmode
\begin{enumerate}
\item This is easy from the construction; it can be found in \cite[proposition 3.4]{Cavalieri2011Polynomial-fami}
\item This is stated in the third bullet point on page 958 of \cite{Cavalieri2011Polynomial-fami} (just above proposition 3.4). %, though we were unable to find the proof. 
\item This is \cite[theorem, top of p9]{Marcus2013Stable-maps-to-}. %, the class associated to the pullback of (*) to $\bar{M}(\cl P/B\bb G_m) \to \bar{M}(\cl P)$ is the double ramification cycle of \cite{Graber2005Relative-virtua}. 
%\item The author was informed of this in a personal communication from Jonathan Wise; Marcus and Wise have a paper in preparation which will present a proof. 
\end{enumerate}
\end{proof}

\begin{remark}[On the characteristic of the base ring]
Write $\bar{M}(\cl P)'$ for the normalisation of $\bar{M}(\cl P)$. In this section we need that $\bar{M}(\cl P)'\to \bar{M}(\cl P)$ is proper, and that $\bar{M}(\cl P)'$ is locally desingularisable. If we work over a field of characteristic zero both conditions clearly hold. More generally, the finiteness of normalisation holds if we work over any universally Japanese base ring, for example $\bb Z$. In general resolution of singularities is more difficult outside characteristic zero. However, in \cite{Marcus2017Logarithmic-com} (which appeared after the first version of the present article) it is shown that $\bar{M}(\cl P)$ is log regular (at least under mild assumptions on the base scheme), hence it is locally desingularisable. With this new reference available, the comparison results in this section are valid over $\bb Z$. 
\end{remark}

Write $\bar{M}(\cl P/B\bb G_m)'$ for the fibre product of $\bar{M}(\cl P)'$ with $\bar{M}(\cl P/B\bb G_m)$ over $\bar{M}(\cl P)$. Because $\bar{M}(\cl P) \to \Mbar$ admits an extension of the section $\sigma$, the same holds for its normalisation $\bar{M}(\cl P)'$, so the latter is $\sigma$-extending. Hence can apply \ref{cor:universal_property_md} to obtain a canonical map $\bar{M}(\cl P)' \to \Md$ (we cannot work directly with $\bar{M}(\cl P)$ since it might not be normal). Recalling that $\on{DRL}_\lozenge$ is the pullback of the unit section $e$ from $J$ to $\Md$, we obtain a diagram 
\begin{equation}\label{dia:two_squares}
\begin{tikzcd}
\bar{M}(\cl P/B\bb G_m)' \arrow[r, "f_\loz"] \arrow[d, tail] & \on{DRL}_\loz \arrow[r] \arrow[d, tail]& \Mbar \arrow[d, tail, "e"]\\
\bar{M}(\cl P)' \arrow[r] &\Md \arrow[r] & J \\
\end{tikzcd} 
\end{equation}
where both squares are fibre products. 

The perfect relative obstruction theory $L_{e/ J}$ for $e \colon \Mbar\tra J$ pulls back to both the other vertical arrows, yielding a class $\kappa$ on $\on{DRL}_\loz$, and a class $D'$ on $\bar{M}(\cl P/B\bb G_m)'$. 

\begin{lemma}
We have an equality $\kappa = \on{DRC}_\loz$ of classes on $\on{DRL}_\loz$. 
\end{lemma}
%The proof essentially consists of unfolding the definitions, and 
\begin{proof}
We begin by recalling the construction of the class $\kappa$ from the relative perfect obstruction theory of $e/J$. The perfect obstruction theory of $e/J$ is the cotangent complex $L_{e/J} = \iota^*\ca I /\ca I^2 [1]$ (where $\ca I$ is the ideal sheaf of $e$ in $J$), which we pull back to $\on{DRL}_\loz$. We then embed in this the relative intrinsic normal cone of $\on{DRL}_\loz$ over $\Md$. This is just the classical cone (as used by Fulton for schemes, and Vistoli for stacks); following \cite{Behrend1997The-intrinsic-n} we are supposed to take the stack quotient  by the relative tangent bundle of $\Md$ over $\Md$, but this is of course trivial. The class $\kappa$ is then obtained by intersecting this cone with the zero section of the cotangent complex (a vector bundle in a single degree). 

Next we compare this to our class $\on{DRC}_\loz$. The key pullback square is
\begin{center}
\begin{tikzcd}
\on{DRL}_\loz \arrow[r, "\sigma'"] \arrow[d, "e'"] & \Md \arrow[d, "e_\loz"]\\
\Md \arrow[r, "\sigma_\loz"] &\Md \times_{\Mbar} J,  \\
\end{tikzcd} 
\end{center}
where both $\sigma_\loz$ and $e_\loz$ are regular closed immersions. To define $\on{DRC}_\loz$, we take $[\Md]$ as a class in the Chow ring of $\Md$, and then apply the refined gysin pullback $\sigma_\loz^!$ to get a class on $\on{DRL}_\loz$ (this is equivalent to taking $e_\loz^!$, see \ref{rem:commutativity}). How is this refined gysin pullback defined? Using that $e_\loz$ is a regular closed immersion, we take its normal bundle in $\Md \times_{\Mbar} J$, which is just the pullback of the normal bundle of $e\colon \Mbar \to J$. We then pull back further to $\on{DRL}_\loz$, and embed the normal cone of $e'\colon \on{DRL}_\loz \to \Md$ inside it, and intersect with the zero section. This is precisely the same as the definition of the class $\kappa$ above, and we are done. 
\end{proof}
%\todo{replace regular embedding with LCI closed immersion, (and define??)}

The above proof was just a matter of checking that Behrend-Fantechi's relative virtual class construction degenerates in our setting to Fulton's intersection product. We could have saved ourselves the effort of writing this proof if we had worked throughout with the language of Behrend-Fantechi, but we wished to emphasise that in our situation this construction is essentially classical.

Now apply Costello's theorem \cite[theorem 5.0.1]{Costello2006Higher-genus-Gr} to the left hand square in the diagram (\oref{dia:two_squares}). The bottom horizontal arrow is birational, and the top arrow is proper since both the source and target are proper over $\Mbar$. Hence we find that ${f_\loz}_* D' = \on{DRC}_\loz$ as classes on $\on{DRL}_\loz$. 

We can make another commutative diagram 
\begin{center}
\begin{tikzcd}
\bar{M}(\cl P/B\bb G_m)' \arrow[r, "f"] \arrow[d, tail] & \bar{M}(\cl P/B\bb G_m) \arrow[r] \arrow[d, tail]& \Mbar \arrow[d, tail, "e"]\\
\bar{M}(\cl P)' \arrow[r] &\bar{M}(\cl P) \arrow[r] & J \\
\end{tikzcd} 
\end{center}
where again both squares are fibre products. Note that $\bar{M}(\cl P)' \to \bar{M}(\cl P)$ is birational, since it is an isomorphism over the locus of smooth curves. Write $D$ for the class on $\bar{M}(\cl P/B\bb G_m)$ arising by pulling back the perfect relative obstruction theory $L_{e/J}$ from $e\colon \Mbar \tra J$. Another application of Costello's theorem yields that $f_*D' = D$. 

\begin{theorem}\label{main_theorem:comparison_full}We have an equality of classes on $\Mbar$:
\begin{equation*}
\on{DRC}_{LGV} = \overline{\on{DRC}}. 
\end{equation*}
\end{theorem}
\begin{proof}
This follows immediately from the above discussion by pushing forward in the commutative diagram
\begin{center}
\begin{tikzcd}
 \bar{M}(\cl P/B\bb G_m)' \arrow[d,swap] \arrow[r, "f"] & \bar{M}(\cl P/B\bb G_m) \arrow[d]\\
\on{DRL}_\loz \arrow[r, swap] & \Mbar. \\ 
\end{tikzcd} 
\end{center}
\end{proof}

\expanded{with all morphisms proper. We have classes:
\begin{itemize}
\item
 $D'$ on $\bar{M}(\cl P/B\bb G_m)'$
 \item 
 $D$ on $\bar{M}(\cl P/B\bb G_m)$
 \item 
 $\on{DRC}_\loz$ on $\on{DRL}_\loz$
 \item $\on{DRC}_{LGV}$ and $\overline{\on{DRC}}$ on $\Mbar$, 
\end{itemize}
and equalities
\begin{itemize}
\item
 ${f_\bloz}_*D' = \on{DRC}_\loz$
 \item 
 $f_*D' = D$
 \item $g_*D = \on{DRC}_{LGV}$
 \item ${g_\bloz}_*\on{DRC}_\loz = \overline{\on{DRC}}$. 
\end{itemize}
The result is immediate from the commutativity of the diagram. }

\Cref{main_theorem:comparison} follows immediately.

\bibliographystyle{amsplain}
\bibliography{prebib.bib}

\end{document}